\documentclass{amsart}
\usepackage{amssymb}
\usepackage{graphicx}
\def\R{{\mathbb R}}
\def\N{{\mathbb N}}

\def\Z{{\mathbb Z}}
\def\1{{1\!\!1}}
\def\E{{\mathbb E}}
\def\P{{\mathbb P}}

\makeatletter

\@addtoreset{equation}{section}
\makeatother

\title{Equivalence of ensembles  for large  vehicle-sharing models}

\author{Christine Fricker}
\address[Christine Fricker]{INRIA Paris-Rocquencourt, Domaine de Voluceau, 78153 Le Chesnay CX, France}
\email{christine.fricker@inria.fr}

\author{Danielle Tibi}
\address[Danielle Tibi]{Universit\'e Paris Diderot, b\^atiment Sophie Germain,  case 7012, 75013 Paris, France}
\email{tibi@math.univ-paris-diderot.fr}


\newtheorem{theorem}{Theorem}[section]
\newtheorem{prop}{Proposition}[section]
\newtheorem{cor}{Corollary}[section]

\newtheorem{lemma}{Lemma}[section]
\newtheorem{rem}{Remark}[section]

\newcommand{\be}{\begin{equation}}
\newcommand{\ee}{\end{equation}}

\begin{document}

\date{\today}

\begin{abstract}
For a class of large closed Jackson networks  submitted to capacity constraints, asymptotic independence of  the nodes  in  normal traffic phase  is proved  at stationarity  under mild assumptions,  using a Local Limit  Theorem. The limiting distributions of  the queues are explicit.  In the  Statistical Mechanics terminology, the equivalence of ensembles - canonical and grand canonical - is proved for specific marginals. The framework includes the case of networks with two types of nodes: single server/finite capacity nodes and infinite servers/infinite capacity  nodes, that can be taken as basic models for bike-sharing systems.  The effect of local saturation is modeled by generalized blocking and rerouting procedures, under which the stationary state is proved to have product-form. The grand canonical approximation can then be used for adjusting the total number of bikes and  the capacities of the stations to the expected demand.
 
\end{abstract}

\maketitle

\section{Introduction}\label{sec1}
 Many cities are  now equipped with bike-sharing systems. Those have met a great success among the public, but they  turn out to face  some difficulty   for adapting to the demand of users. One major  inherent problem is due to  the   spatial  heterogeneity  of the demand. It results in the common observation that, during the day,  some parking stations stay most of the time empty, while others keep many vehicles unused.  Some sites may indeed be more popular  for picking a vehicle than for returning it.  This induces two possible drawbacks for users, namely,  finding  no vehicle for rent at their starting point, or  no parking room  at the end of their ride. Indeed, since parked vehicles are locked, each station has only a given number of docking places.  This capacity limitation may be most problematic, making customers loose time in an additional  run (and  then  possibly a walk back to their former  destination)   just for parking their vehicle. On the contrary, finding no vehicle at some station is more simply resolved - using alternative means of  transport. A major concern, when designing such a network, is thus to find the best acceptable tradeoff, relating to the total number of  offered vehicles,  as function of the different capacities of the stations. A low supply of vehicles produces empty stations, while a large supply causes  saturation.

  \smallskip

  The purpose of the present paper is to analyze large, inhomogeneous networks  with finite capacity stations.  The performance of such models is estimated through explicit asymptotics of their stationary state. Performance indicators such as the total rate of failure can then be optimized by adapting the flexible parameters (the total number of vehicles  and, possibly, the sizes of the stations).   Note, yet, that this  may not be sufficient to reach a good quality of service. Ultimately, different devices should be implemented in order to  limit service failure, or disparity between stations: For  example, setting up a vehicle reservation service, or  supplying empty stations with vehicles (moved from saturated stations),  or  else displaying   online updates of the current state of the network (so that one knows where parking room or vehicles are available). In this regard, besides quantifying  the weaknesses of the simplest  primitive system, we introduce  state-dependent routing procedures that may open the way to such improvements. 

 \smallskip

  Time inhomogeneity  is  another concern: The behavior of users should vary in time,  more or less obeying a 24 hours cycle. This issue  will not be addressed in this paper, where all processes will be assumed to have  constant traffic parameters and to have reached their stationary regime. In this respect, a major issue would be  to estimate  the relaxation time  -  or time to reach equilibrium - of such  processes, and  compare it  with the duration of the constant parameter phase of interest.
  
  \smallskip
   It has been commonly observed  that  vehicle-sharing systems  can be modeled by closed Jackson networks. In this description, the vehicles  play the role of the customers  - which number is fixed - while the  users act as successive servers at the parking stations. The associated network has two types of nodes: one server nodes, that describe the stations, and infinite server nodes, representing  the different routes between the stations. In  \cite{Fayolle-7} a detailed  analysis of large, infinite capacity, closed Jackson  networks  at equilibrium is proposed. Application to vehicle sharing systems - at their early experimental stage in the 90's  -  is briefly considered.   The asymptotics of  a model  mentioned in  \cite{Fayolle-7} is examined in  \cite{George-1} as the number of nodes (stations and routes) is fixed, while the number of customers (vehicles) grows to infinity. Both  papers crucially rely on the explicit product-form of the stationary distribution, which is standard in the infinite capacity case.

    In \cite{FrickerGast-velib},  assuming complete homogeneity of the traffic and equal finite  capacities  of the stations,  a mean-field asymptotic behavior is obtained for the  dynamics of the distribution of vehicles over the different  sites.  Using the same method, \cite{Heterogeneous} extends the study  to systems where inhomogeneity is modeled by clusters. In addition, different alternatives are investigated, as  for example, avoiding empty (respectively full) stations when taking  (respectively returning) a bike,  or  returning bikes at the station  having  the largest available space,  between two randomly chosen stations.

The present paper analyzes the impact of finite capacity on large inhomogeneous networks.  The models considered thus involve real-world blocking and rerouting mechanisms, that we implement on the infinite capacity model of  \cite{Fayolle-7} and  \cite{George-1}. The crucial product-forms of the stationary states are  proved for  a class of state-dependent routings that extends the  classical setting of \cite{Economou-1}.  Asymptotics are obtained as both the number of vehicles and the size of the network increase. Namely, it is  proved that the  finite dimensional distributions of the stations occupation process at stationarity converge to products of truncated geometric distributions. If the total number of bikes does not exceed some order fo magnitude, stations and routes are moreover asymptotically independent, the latter being  approximately distributed as Poisson variables.  The approach, at stationarity,  is similar to  \cite{Fayolle-7} and  \cite{George-1}.
 As compared to  the dynamical  point of view of  \cite{FrickerGast-velib} and \cite{Heterogeneous}, our results complete the  picture of the equilibrium state  there obtained for locally homogeneous systems.

 \smallskip
 Our asymptotics fit the frame of the so-called principle of equivalence of canonical and grand canonical ensembles,  in the Statistical Mechanics terminology.  The networks of interest  are here  in some ``partially  subcritical" phase, in the sense that  part of  the network - namely, the parking stations - is in normal traffic (or no-condensation) regime. This makes  that a Local Limit Theorem  can be used.
This way of deriving the grand canonical approximation is standard and can be traced back to Khinchin \cite{Khinchin-1}.  Dobrushin and Tirozzi \cite{Dobrushin-1}  have used it for analyzing a family of Gibbs measures on $\Z^2$. It has, since  then, proved useful in many different contexts, notably for exclusion  and zero-range processes under thermodynamic limit (\cite{Kipnis-1} Appendix 2).  Note that zero-range processes are identical to (homogeneous) closed Jackson networks,  but due to homogeneity, condensation can only occur in specific attractive cases, which are not considered in  \cite{Kipnis-1}. Such supercritical phases are described in  \cite{Grosskinsky-1}  and  \cite{Armendariz-1}.  The generalized Jackson networks with blocking considered in Section \ref{sec4}  contain as particular cases the generalized exclusion processes analyzed in \cite{Kipnis-1}.

 Contrary to the classical setting of the equivalence of ensembles - or its alternative formulation as Gibbs conditioning principle (\cite{Dembo-1}) - the equilibrium states of  general  closed Jackson networks  are non-symmetrical with respect to the nodes.  Malyshev and Yakovlev (\cite{Malyshev-6}) have used analytical methods to address the  problem  for  classical (one-server nodes) closed Jackson networks, under existence of a limiting density of customers -  as in  \cite{Kipnis-1} -  and of a limiting profile, capturing inhomogeneity, for the  traffic parameters. The critical phenomenon of condensation is highlighted and both  phases  are studied. These results are extended in   \cite{Fayolle-7}, using Local Limit Theorems, which is more flexible and does not require specification of a limiting density nor parameter profile.   
Note yet  that this method does not lead  to a complete and unified treatment of supercritical regimes, for which  restrictive assumptions  are needed.

We use the same approach and for simplicity,  restrict the analysis to the  relevant dynamics  in respect to the targeted  vehicle-sharing applications. Though  our interest focuses on finite capacity systems, the infinite capacity case - that could equally be derived from \cite{Fayolle-7} - is included,  as a reference and  for completeness. Similarly,  several standard results (as regards the Local Limit Theorem, or  generalized Jackson networks) are here stated, under specific assumptions adapted to the present context,   for the sake of self-containment.

 The paper is organized  as follows. Section \ref {sec2} displays the Central Limit and  Local Limit  Theorems - for independent, non identically distributed, random variables -  that will be used in the sequel.  Section  \ref {sec3} introduces the generalized Jackson processes of interest and states the  product-forms  of the associated  stationary distributions. In Section  \ref {sec4},  the   equivalence of ensembles is proved for two classes of  networks (with respectively infinite and finite capacities).  Applications to evaluating the performance of bike-sharing systems is finally developed in Section \ref {sec5}.

\section{A Local Limit Theorem}\label{sec2}
It is classical that the Central Limit Theorem holds for any sequence $(S_N)_{N \ge 1}$ of sums of independent square integrable random variables  $(X_{j,N})_{1 \le j \le J(N)}$
\[ S_N = \sum_ {j=1}^{J(N)} X_{j,N}\]
if the family  $(X_{j,N})_{1 \le j \le J(N)}$   satisfies the following so-called {\it Lyapunov condition}:
\begin{equation}\label{Lyapunov}
\exists \delta >0 \   \text{ such that } \    \lim _{N \to \infty}  \frac{1}{b_N^{2+\delta}} \sum _{j=1}^{J(N)} \E \left ( |X_{j,N} - m_{j,N}|^{2+\delta}\right ) =0.
\end{equation}
Here and in the sequel, for all $N \ge 1$, $J(N)$ is some integer such that $J(N) \ge 1$, and we use the following notations: 
\begin{multline*}\label{moments}
 m_{j,N} = \E(X_{j,N}) \quad (1 \le j  \le J(N)),  \qquad    a_N = \E(S_N) = \sum_{j=1}^{J(N)} m_{j,N}    \qquad  \text { and } \\ 
  \sigma^2_{j,N} = \E ([X_{j,N}-m_{j,N}]^2) \ , 
  \qquad   b_N^2 = \E ([S_N -a_N]^2) =  \sum _{j=1}^{J(N)} \sigma^2_{j,N}. 
 \end{multline*}
This result is usually derived from the Lindeberg Central Limit Theorem, which states that  under the following {\it Lindeberg condition} \eqref{Lindeberg}, which is weaker than \eqref{Lyapunov}, the characteristic function (or Fourier transform) of  $b_N^{-1}(S_N -a_N)$ converges to that of the standard normal  distribution (see  \cite{Billingsley3}, or \cite{Feller-2}):
\begin{equation}\label{Lindeberg}
\forall \varepsilon >0, \qquad  \lim _{N \to \infty}  \frac{1}{b_N^{2}} \sum _{j=1}^{J(N)}  \E \left ( (X_{j,N} - m_{j,N})^{2} \cdot  {\bf 1}_{|X_{j,N} - m_{j,N}| > \varepsilon b_N} \right ) =0.
\end{equation}

 It is easily seen that, under the Lyapunov condition, the convergence is uniform on some sequence of intervals $[-A_N, A_N]$ which length grows to infinity. This refinement of the Central Limit Theorem is stated below as {\it Lyapunov Central Limit Theorem}. It will be used for deriving the next {\it Local Limit Theorem} that is the key to the equivalence of ensembles. 
 
 The proofs of both theorems, as well as those of the two next propositions are deferred to the Appendix.

\medskip

Let   $(X_{j,N})_{ N \ge 1, 1 \le j \le J(N)}$ be a family of square integrable random variables  such that, for all $N \ge 1$,   $X_{1,N},  \dots , X_{J(N),N}$ are independent. Then define  $S_N, a_N$ and $b_N$  as above.

\begin{theorem}\label{TCL} {\bf (Lyapunov Central Limit Theorem)}

If the Lyapunov condition \eqref{Lyapunov} is satisfied, then the following convergence holds,
\begin{equation}\label{uniform}
 \lim _{N \to \infty} \sup _{\ | t |  \le  A_N}\left |  \, \E \left ( e^{ it \frac{S_N-a_N}{b_N}}\right ) -  e^{-t^2/2} \, \right | = 0,
  \end{equation}
for some sequence $(A_N)$ of positive real numbers converging to infinity.

\end{theorem}

All the models of interest in this paper will satisfy the Lyapunov condition with $\delta =1$.

\medskip
Deriving of the Local Limit Theorem  from the Central Limit Theorem  for integer valued random variables has become standard since the  paper by Gnedenko \cite{Gnedenko-1}, addressing the case of i.i.d.~variables. We here give a version adapted to our context. 

\begin{theorem}\label{TLL} {\bf (Local Limit Theorem)} 

Assume that the random variables $X_{j,N}$ are integer $\Z$-valued and that the following conditions hold: 
\begin{enumerate}
\item $\lim _{N \to \infty} b_N = + \infty$,
\item  there exists some  sequence $(A_N)$ of positive real numbers converging to infinity  such that \eqref{uniform}  holds, 
\item  there exists some $ \phi \in L^1(\R)$ such that for all $N \ge 1$  and all $t \in [-\pi, \pi]$,
\[  | \,  \E(e^{it S_N}) \, | \le \phi(b_N t).\]
\end{enumerate}
Then the following is true
\begin{equation} \label{tll} \lim _{N \to \infty} \  \sup _{k \in \Z} \left [ b_N \sqrt{2 \pi} \,  \P(S_N =k) -  \exp \left (-\frac{(k-a_N)^2}{2 b^2_N} \right )\right ] ~=~0.
\end{equation}

\end{theorem}

We will now consider two examples of interest for our applications of Section \ref{sec4}. Conditions (i) of the following Propositions \ref{geompoisson} and \ref{trunkgeompoisson}  will  there mean  that  the networks  are subcritical as concerns their one-server nodes. 

\bigskip

{\bf Example 1.}
\vspace{2mm}

 Suppose here that, for all $N$, the set of indices $\{ 1, \dots , J(N)\}$ is the union of two disjoint sets $\mathcal J^1_N$ and $\mathcal J^2_N$, such that

\vspace{2mm}
 $ \bullet$ for $j \in  \mathcal J^1_N$, the random variable $X_{j,N}$ is geometric with parameter $\rho_{j,N} \in [0,1[$, 
 \[ i.e. \quad  \P(X_{j,N} = k ) = (1-\rho _{j,N})  \rho _{j,N}^k, \quad \text{for}\   k \in \N,\]

\vspace{1mm}
$ \bullet$  for $j \in  \mathcal J^2_N$, the random variable $X_{j,N}$ is Poisson with parameter $\lambda_{j,N}  >0$~,
 \[  i.e. \quad  \P(X_{j,N} = k ) = e^{-\lambda _{j,N}}   \frac {\lambda _{j,N}^k}{k!}, \quad \text{for} \  k \in \N.\]
 
The means and variances of the $X_{j,N}$'s and $S_N$'s are then given by
\begin{multline*}
 m_{j,N} = \left \{ \begin{array}{ll} \frac{\rho _{j,N}}{ 1-\rho _{j,N}}  & \text{ for } j\in \mathcal J_1^N \\  \lambda_{j,N}  & \text{ for } j\in \mathcal J_2^N \end{array},\right.  
     \qquad  \sigma^2_{j,N} =  \left \{ \begin{array}{ll} \frac{\rho _{j,N}}{ (1-\rho _{j,N})^2} & \text{ for } j\in \mathcal J_1^N \\  \lambda_{j,N} &  \text{ for } j\in \mathcal J_2^N  \end{array},\right.    \\ 
     \\   
     a_N =   \sum _{j \in \mathcal J_1^N} \frac {\rho _{j,N}}{ 1-\rho _{j,N}} +  \sum _{j \in \mathcal J_2^N} \lambda _{j,N}   \;\;   \text{ and } \;\;  b_N^2 =\sum _{j \in \mathcal J_1^N} \frac{\rho _{j,N}}{ (1-\rho _{j,N})^2} +  \sum _{j \in \mathcal J_2^N} \lambda _{j,N}  . 
 \end{multline*}
 
\smallskip

\begin{prop}\label{geompoisson}
 If the two following conditions are satisfied:

(i) there exists some $\rho <1$ such that  for all $ N\ge 1 $ and  $j \in \mathcal J^1_N$, $~\rho _{j,N} \le \rho$,

(ii) $\lim _{N \to \infty} b_N= + \infty$,

 then \eqref{tll} holds. 

\end{prop}

\begin{rem}  $(i)$ and $(ii)$  of Proposition \ref{geompoisson} are not  necessary conditions: The Local Limit Theorem can  still hold  beyond  the subcritical phase.  It can be proved  for example  that if the number of indices
$j \in \mathcal J^N_1$ such that $ \rho_{j,N} = \max _{j \in \mathcal J_1^N} \rho _{j,N}$ goes to infinity  and $\ \liminf  \max _{j \in \mathcal J_1^N} \rho _{j,N} >0$, then equations \eqref {uniform} and \eqref {tll}  hold.
\end{rem}

\smallskip

{\bf Example 2.}
\vspace{2mm}

 Suppose now that, for all $N$, the set $\{ 1, \dots , J(N)\}$ is the union of two disjoint sets $\mathcal J^1_N$ and $\mathcal J^2_N$, such that
 
 \begin{itemize}
 \item for $j \in  \mathcal J^1_N$, the random variable $X_{j,N}$ has truncated geometric distribution with parameters $\rho_{j,N} \ge 0$ and $ c_{j,N} \in \N $, i.e. for $   k \in \{0, \dots , c_{j,N}\}$, 
 \[   \quad \P(X_{j,N} = k ) = \frac{ \rho _{j,N}^k} {\sum _{h=0}^{c_{j,N}}\rho _{j,N} ^h}  , \]
 \item  for $j \in  \mathcal J^2_N$, the random variable $X_{j,N}$ is Poisson with parameter $\lambda_{j,N} >0$ as in example 1.
\end{itemize}

The means and variances of  $S_N$'s are here given by
\[  a_N =   \sum _{j \in \mathcal J_1^N} m _{j,N} +  \sum _{j \in \mathcal J_2^N} \lambda _{j,N}   \qquad   \text{ and } \qquad   b_N^2 =\sum _{j \in \mathcal J_1^N} \sigma ^2 _{j,N}+  \sum _{j \in \mathcal J_2^N} \lambda _{j,N}, \]
where for $ j\in \mathcal J_1^N$,
 $\quad \displaystyle {m_{j,N} = \left \{ \begin{array}{ll} \frac{\rho _{j,N}}{ 1-\rho _{j,N}}  - (c_{j,N} +1)\frac{\rho _{j,N}^{1+c_{j,N}}}{ 1-\rho _{j,N}^{1+c_{j,N}}} & \text{ if } \rho _{j,N} \neq 1\\  \frac{c_{j,N}}{2}  & \text{ if }\rho_{j,N} =1\end{array}\right. } $ 
  
\[\text{ and} \qquad  \sigma^2_{j,N} =  \left \{ \begin{array}{ll} \frac{\rho _{j,N}}{ (1-\rho _{j,N})^2} - (c_{j,N} +1)^2\frac{\rho _{j,N}^{1+c_{j,N}}}{ \left (1-\rho _{j,N}^{1+c_{j,N}}\right )^2} & 
      \text{ if } \rho _{j,N} \neq 1\\  \frac{c_{j,N}(c_{j,N} +2)}{12}  & \text{ if }\rho_{j,N} =1
      \end{array}\right.    \]
 
\smallskip

\begin{prop}\label{trunkgeompoisson}
If the two following conditions are satisfied:

(i) there exists some $C < +\infty$ such that for all $ N \ge 1$ and $j \in \mathcal J^1_N$, $~c _{j,N} \le C,$ 

(ii) $\lim _{N \to \infty} b_N = + \infty$,

 then \eqref{tll} holds.  

\end{prop}

The proof of this proposition (see Appendix) requires  uniform domination of the  characteristic functions of truncated geometric variables. This is stated as  Lemma \ref{domination-Gnedenko}, which itself relies 
on the following lemma, due to Gnedenko \cite[p.~192-3]{Gnedenko-1}. 
\smallskip
\begin{lemma} \label{gnedenko} {\it (Gnedenko)}
Let $X$ be an integer-valued random variable. Denote $p_n= P(X=n)$ for $n \in \Z$ and  
 \[ s = \sum _{l \in \Z} \frac{p_{2l}p_{2l+1}}{p_{2l}+p_{2l+1}}, \]
  with the convention that  $\frac{p_{2l}p_{2l+1}}{p_{2l}+p_{2l+1}} =0\ $ if $\ p_{2l}  =p_{2l+1}= 0$.
  Then, the following inequality holds
\begin{equation} \label{domination}
 |\, \E(e^{itX}) \,  | ~\le~ e^{-\frac{2}{\pi ^2} s  t^2} \qquad \text{ for }  \  |t| \le \pi.
  \end{equation} 
\end{lemma}

\medskip
Now  for any  $\rho >0$ and  $c \in \N$,  let  $m_{\rho, c}$ and $\sigma ^2_{\rho , c} $ denote the mean and variance of a truncated geometric random variable  $X_{\rho, c}$ with parameters $\rho$ and $c$, and  let $s _{\rho, c}$ denote the associated parameter as defined in Lemma \ref{gnedenko}.  The following lemma is easily derived from Lemma \ref{gnedenko}.

\begin{lemma}\label{domination-Gnedenko} 
For any $C \ge 1$, there exists some $\kappa>0$  such that for any  $\rho > 0$ and any  $c \in \{0, \dots , C\}$,
\[  |\, \E(e^{itX_{\rho, c}}) \,  | ~\le~ e^{-\frac{2}{\pi ^2} \kappa \sigma_{\rho, c}  ^2  t^2} \qquad \text{ for }  \  |t| \le \pi.\]
\end{lemma}

\begin{proof} It is enough to consider the case when $0 <\rho \le 1$, since $ \ \sigma ^2_{\rho^{-1},c}=\sigma ^2_{\rho , c}$ and 
\[  |\, \E(e^{itX_{\rho, c}}) \,  | = |\, \E(e^{it(c-X_{\rho ^{-1}, c}}) \, |  = |\, \E(e^{-it X_{\rho ^{-1}, c}}) \, |  \] 
 due to the following elementary remark.
 
\begin{rem}\label{symmetry}
For all $c\in \N$ and $\rho\in ]0,+\infty[$, the random variable $  X_{\rho, c}$ has the same distribution as  $c-X_{\rho^{-1},c}$.
   As a consequence, the following symmetries hold for any $c \in \N$ and $\rho > 0$:  
   \[m_{\rho^{-1}, c} = c-m_{\rho, c} \qquad \text{  and  } \qquad \sigma ^2_{\rho^{-1},c}=\sigma ^2_{\rho , c} . \]
This induces a duality  property  between occupied and empty rooms in the finite capacity systems to come,  known as the \emph {particles-holes duality  (see \cite{Kipnis-1})}.
   \end{rem}

   So, it is now enough to prove existence of some positive $\kappa $ such that 
   \[ s_{\rho, c} \ge \kappa \sigma^2_{\rho, c}\]
for all $c \in \{1, \dots , C\}$ and $0 <\rho \le 1$.  (The case $c=0$ needs not be considered, since $s_{0,\rho}= \sigma ^2_{0, \rho} =0$ for all $\rho$.) 

It is clear that for fixed $c \ge 1$, both mappings  $\rho \mapsto s _{ \rho, c}$ and $\rho \mapsto \sigma^2_{\rho, c}$ are positive valued and continuous on the interval $]0, 1]$. It is then enough to prove that,  
  for $   c= 1, \dots , C$, 
$$ \liminf _{ \rho \to 0} s _{ \rho, c} / \sigma^2_{\rho, c} >0. $$
Now from the definition of $s _{\rho, c}$, one gets (using $\rho \le 1$ for the second inequality)
\begin{align*}
  s _{\rho, c} &= \left (\sum _{l=0}^c \rho ^l \right )^{-1} \left (\sum _{0 \le 2l \le c-1}  \frac { \rho ^{4l+1}}{ \rho ^{2l} + \rho ^{ 2l+1}} \right ) \\
&~\ge~\frac{ \rho}{2} \left (\sum _{0 \le 2l \le c-1} \rho ^{2l}\right ) \left (\sum _{l=0}^c \rho ^l \right )^{-1} ~\ge~\frac{ \rho}{2} \left (\sum _{l=0}^c \rho ^l \right )^{-1}.
\end{align*}
Together with $\sigma^2_{\rho, c}  = \frac{\rho}{ (1-\rho )^2} - (c +1)^2\frac{\rho ^{c+1}}{ \left (1-\rho ^{c+1}\right )^2} $
if $\rho<1$, we obtain
\[ \liminf _{ \rho \to 0} \frac{s _{ \rho, c} }{ \sigma^2_{\rho, c}}  \ge \liminf _{ \rho \to 0} \left [ 2 \left (\sum _{l=0}^c \rho ^l \right ) \left ( \frac{1}{ (1-\rho )^2} - (c +1)^2\frac{\rho ^{c}}{ \left (1-\rho ^{c+1}\right )^2} \right )  \right ]^{-1} =~ \frac{1}{2}. \] 
\end{proof}

\section{Generalized closed Jackson networks}\label{sec3}

In accordance with the terminology of Serfozo~\cite{Serfozo-1}, a standard Jackson network consists of a finite number of nodes at which customers are successively served, according to possibly different procedures. Nodes operate independently, which means  that  the  rate at which customers leave a node only  depends on the  current  occupation of this node. When a customer completes his service at  some node, he  instantly moves to another node for another service. New nodes are chosen at random  according to a fixed routing matrix. All service processes and  routings are independent. For a \emph {closed} Jackson network,  which is the case considered in this paper, customers stay forever in the system and there are no external arrivals.

In this description all nodes have unlimited capacity. This dynamics can be generalized to a system including   nodes with finite capacities. But the transition rule needs then  be state dependent  (constrained by the current location of the free space). It can no longer be given by some fixed matrix. Different blocking or rerouting policies can be considered.

\subsection{Infinite capacity model} \label{classicalJackson}
First  consider the case when all node  capacities are infinite. Denote by $N$ the number of nodes, by $M$ the fixed number of customers and by $P= (p_{ij})_{1\le i,j \le N}$ the routing matrix (that is, $p_{ij}$ is the probability that a customer leaving node $i$ moves to node $j$).  $P$  is assumed   irreducible, which ensures uniqueness of its associated invariant distribution $\theta= (\theta _1, \dots , \theta _N)$ on $\{ 1, \dots , N\}$.  For each node $i$, the departure rate of customers from node $i$ is $g_i(k)$ when  $k$ customers are present at $i$. Here $g_i$ is some function defined on $\N$ such that $g_i(0)=0$ and $g_i(k) >0$ for $k \ge 1$.

 The node-occupation process is Markov with state space 
 \[ \mathcal S_{N,M} =  \left \{ (n_1, \dots , n_N) \in \N^N \quad \text{  such that } \quad  \sum _{j=1} ^N n_j = M \right \}, \]
and transitions
\[ n \longrightarrow n-e_i+e_j \quad \text{ at rate } \quad q(n,n-e_i+e_j ) = g_i(n_i) p_{ij},\]
for $n \in \mathcal S_{N,M}$ and $ \ 1 \le i,j \le N$. Here $e_i$ denotes the $i^{th}$ unit vector. 

 A well known and  remarkable feature of this class of processes is that  the joint stationary distribution of the queues lengths $\xi _1, \dots , \xi _N$ is explicitly known and has  product form, given  by
 \begin{equation}\label{invariant}
 \P(\xi _1= n_1, \dots , \xi _N =n_N) =  \frac{1}{Z} \prod _{j=1} ^N \frac{\theta_{j}^{n_j}}{(g_j !) (n_j)} \qquad \text{ for } (n_1, \dots , n_N)  \in \mathcal S _{N,M} ,
 \end{equation}
 where $Z$ is a normalizing constant and for $j = 1, \dots , N$, the function $(g_j !) $ is defined  on  $ \N$ by
 \[  (g_j !) (0)  = 1 \quad \text { and } \quad (g_j ! )(n) = \prod _{k=1}^n g_j  (k) \quad \text{ for } n \ge 1. \]

\subsection{Finite capacity models} \label{subsec3.2}
 If in the contrary some nodes have finite capacities, denote  $c_i$ the capacity of node $i$ ($0 \le c_i \le + \infty $ for $i = 1, \dots , N$) and  $c$ the vector $(c_1, \dots , c_N)$.   The state space is  now changed for
  \[ \mathcal S^{c}_{N,M} =  \left \{ (n_1, \dots , n_N) \in \N^N  \  : \ \   \sum _{j=1} ^N n_j = M    \quad \text{and}   \quad   n_j \le c_j \   \text { for }  c_j <+ \infty 
  \right \}, \]
  which is non empty provided that  $ \  \displaystyle {\sum _{j=1} ^N c_j  \ge  M  }$ .
 Several  dynamics are considered.
 
\medskip 
{\it 1 - Model with blocking.}

\smallskip 
The dynamics is the same as in the  previous infinite capacity case, except that when node $j$ is at capacity, any move  from some node $i$ to node $j$ is cancelled   (that is, the customer involved renews service at node $i$). 

 In other terms, the transitions are given by
\[ n \longrightarrow n-e_i+e_j \quad \text{ at rate } \quad q(n,n-e_i+e_j ) = g_i(n_i) p_{ij}  {\bf 1} _{n_j< c_j},\]
for $n \in \mathcal S^{c}_{N,M}$ and $ \ 1 \le i,j \le N$.

  The infinite capacity model can be seen as a particular case of this finite capacity model with blocking.  But the explicit formula \eqref{invariant} for the invariant distribution does not extend, in general, to the finite capacity model with blocking. It is the case, yet,  if the routing matrix $P$ is \emph {reversible} with respect to its invariant measure $\theta$, that is, if
   \begin{equation} \label{reversible} \theta _i p_{ij} = \theta _j p_{ji}  \quad \text{ for } 1 \le i,j \le N . 
   \end{equation}
 If \eqref{reversible} is satisfied, then  \eqref{invariant} still holds here for $(n_1, \dots , n_N)  \in \mathcal S^{c} _{N,M}$ (see \cite{Serfozo-1}).
 
  \bigskip 
{\it 2 - Model with blocking and rerouting.}

 \smallskip
 This is a particular case of a dynamics  introduced by Economous and Fakinos \cite{Economou-1}. Here, any customer routed - according to the matrix $P$ - to some saturated node  will move instantly (i.e., at infinite speed)  across the system, still according to $P$, up to find some non saturated node where to settle for a new service. Note that this customer will certainly find such a node, possibly the actual node he has just leaved (where one unit of capacity is  thus available). Of course, in this case, no transition occurs. 
 
 The corresponding  process has the following transitions and rates
\begin{equation} \label{rates}  n \longrightarrow n-e_i+e_j \quad \text{ at rate } \quad q(n,n-e_i+e_j ) = g_i(n_i) p^*_{ij}(n) ,
\end{equation}
for $n \in \mathcal S^{c}_{N,M}$ and $ \ 1 \le i \ne j \le N$,  where $ p^*_{ij}(n)$ denotes the probability that a Markov chain with transition matrix $P$ initiated at $i$ enters the   set $A_i(n) $ at $j$, for
\[ A_i(n)  \stackrel{def}{=} \{ j, \ 1 \le j \le N, \quad j=i \quad \text{or} \quad n_j <c_j \}. \]

Contrary to the simple blocking case where $P$ has to be assumed reversible,  here the product form \eqref{invariant} holds  \emph {without any restriction on the routing matrix $P$}.  This dynamics thus appears as the appropriate generalization of the standard - infinite capacity - Jackson dynamics (in which case,  $ p^*_{ij}(n) =  p_{ij}$ for all $n,i,j$).

\begin{theorem} \label{Economou}  (Economou and Fakinos).
For any irreducible routing matrix $P$, the Markov process with state space $\ S^{c} _{N,M}$ and transition rates given by \eqref{rates} has invariant distribution  given by  \eqref{invariant}  for $(n_1, \dots , n_N)  \in \mathcal S^{c} _{N,M}$.
\end{theorem}

 \bigskip 
{\it 3 - Other state-dependent routings.}

 \smallskip
The scheme introduced in  \cite{Economou-1} is slightly more general than the blocking and rerouting dynamics just described. In the general setting of \cite{Economou-1}, the set $A_i(n)$ - of nodes that are below capacity when some customer has just left  node $i$ - is splitted into disjoints blocks. A partition is thus associated with each configuration $n-e_i$ of $M-1$ customers over the $N$ nodes. The just served customer  at $i$ still explores the network at infinite speed according to some Markov chain with transition $P$, but    settles at the first visited node  that  \emph {belongs to the same block} $B_{n-e_i}(i)$  as $i$. The same product form stationary distribution then holds, given by  \eqref{invariant}.

This result can be generalized to a larger class of state-dependent routings. We consider the following  transitions and rates:
\begin{equation} \label{newrates}  n \longrightarrow n-e_i+e_j \quad \text{ at rate } \quad q(n,n-e_i+e_j ) = g_i(n_i) p_{ij}(n-e_i) \, {\bf 1}_{n_i >0,\, n_j <c_j} ,
\end{equation}
for $n \in \mathcal S^{c}_{N,M}$ and $ \ 1 \le i \ne j \le N$,  where   a transition matrix  $ P(m) =(p_{ij}(m)) _{ i,j \in A(m)}$ on the set
\[ A(m)  \stackrel{def}{=} \{ j, \ 1 \le j \le N, \quad  m_j <c_j \} \]
is associated with each $m \in \mathcal S^{c}_{N,M-1}$.

\begin{theorem} \label{state_dependent} 
Assume that some positive vector $\theta = (\theta _1, \dots , \theta_N)$ is such that
\[ \forall m \in  \mathcal S^{c}_{N,M-1},  \quad \forall j \in \{1, \dots , N\} \quad \sum _{i \in A(m)} \theta _i p_{ij}(m) = \theta _j.\]
then the Markov process with state space $\mathcal S^{c} _{N,M}$ and transition rates given by \eqref{newrates} has invariant distribution  given by  \eqref{invariant}  for $(n_1, \dots , n_N)  \in \mathcal S^{c} _{N,M}$.
\end{theorem}

The dynamics considered in \cite{Economou-1} are of  the form described in \eqref {newrates}, with line $i$ of matrix $P(m)$  then given, for all $m \in \mathcal S^{c}_{N,M-1}$  and $i \in A(m)$, by the distribution of the return point to $B_{n-e_i}(i)$ of a Markov chain  with transition matrix $P$ initiated at $i$. The condition of Theorem  \ref{state_dependent} is then satisfied for  the invariant vector $\theta$ of $P$, due to the well-known following result:  For any ergodic Markov chain on $\{1, \dots ,N \}$ with invariant vector $(\theta _i)_{1 \le i \le N}$, and for any subset $B$ of $\{1, \dots ,N \}$, the embedded Markov chain at times of visits to $B$  has invariant vector   $(\theta  _i)_{i \in B}$.

\begin{proof}  We check  that the global balance equations 
\begin{align}\label{balance}
 \sum_{\substack{ 1\le i \ne j \le N \\ n_i >0, n_j <c_j}} g_i(n_i)   p_{ij}(n-e_i)  = \!  \sum_{\substack{ 1\le i \ne j \le N, \\    n_i>0 , \, n_j <c_j} } \! g_j(n_j+1)  p_{ji}(n -e_i)  \frac{\pi(n-e_i+e_j)}{ \pi (n) } ,
\end{align} 
 are satisfied for  $n \in \mathcal S^{c}_{N,M}$, where  $\pi (n) =  Z^{-1} \prod _{j=1} ^N \theta_{j}^{n_j}/(g_j !) (n_j)$.

Recall that $g_i(n_i)=0$  if $n_i=0$, so that the left-hand side of  \eqref{balance} is equal to 
\[  \sum_{i=1}^ N g_i(n_i)  \left ( 1- p_{ii}(n-e_i) \right ), \]
while, using  $ \pi(n-e_i+e_j) / \pi (n) = \theta _j g_i(n_i) / (\theta  _i g_j(n_j+1))$,  the right-hand side rewrites
\[  \sum_{\substack{ 1\le i \ne j \le N\\   n_j <c_j }} \frac{\theta  _j}{ \theta  _i} g_i(n_i)  p_{ji}(n -e_i)  =   \sum_{i=1}^N  \frac{g_i(n_i) }{ \theta  _i}   \sum_{  j \ne i,  \, n_j <c_j }  \theta  _j  p_{ji}(n-e_i) \]
 Now equation \eqref {balance}  results from the following equality: for all $n \in \mathcal S^{c}_{N,M}$ and $i \in \{1, \dots , N \}$ such that $n_i >0$, 
  \[  \theta  _i  \left ( 1- p_{ii}(n-e_i)  \right ) =   \sum_{  j \ne i,  \, j \in A(n-e_i) }  \theta  _j  p_{ji}(n-e_i), \]
 or equivalently (adding $\theta  _i   p_{ii}(n-e_i)$ to both members)
  \[  \theta  _i   =   \sum_{  j \in A(n-e_i) }  \theta  _j  p_{ji}(n-e_i), \]
  which is satisfied by assumption on vector $\theta$.
 
 \end{proof}

 \begin{rem}\label{uptoaconstant} 
 It will next be crucial to note that  due to the particular form of the constraint $\sum _{j=1} ^N n_j = M$, the expression   in \eqref{invariant},   for $(n_1, \dots , n_N) $ in $\mathcal S _{N,M}$ or  $\mathcal S^{c} _{N,M}$,  is unchanged if the invariant vector $\theta$ is replaced by some proportional vector   $\gamma  \theta =  (\gamma   \theta _1,\ldots, \gamma   \theta _N)$,  where $\gamma >0 $ is arbitrary.
 \end{rem}

\section{Equivalence of ensembles for generalized Jackson networks}\label{sec4}
 
 The product formula  \eqref{invariant}  can be interpreted as follows: the invariant joint distribution of the $N$ queues is the distribution of $N$ independent random variables   conditioned to the constraint that their sum is $M$. 
 
 Moreover,  for  \emph {arbitrary positive} $\gamma$, remark ~\ref{uptoaconstant} allows  one to  choose  the individual distributions of these independent variables, denoted  $ \eta ^{\gamma}  _1, \dots ,  \eta ^{\gamma}  _N$, as given by
 \begin{equation} \label{individual} \P(  \eta ^{\gamma}  _j = n) = \frac{1}{Z_j(\gamma)} \, \frac{(\gamma   \theta_{j})^{n}}{(g_j !) (n)} \qquad \text{ for }\,  0 \le n  < c_j +1,  \quad j= 1, \dots , N, 
 \end{equation}
 provided that $\gamma$ belongs to the domain of convergence of each  power series  \begin{equation}\label{series}
 Z_j(\gamma) = \sum _{n=0}^{  c_j} \frac{  \theta_{j}^{n}}{(g_j !) (n)}   \gamma ^{n} \qquad \text{ for }   \quad j= 1, \dots , N.  
 \end{equation} 
 For each such $\gamma$,   \eqref{invariant}  rewrites:   for  $(n_1, \dots , n_N)  \in \mathcal S _{N,M} $,
 \begin{equation}  \label{conditionning} \P \left (\xi _1= n_1, \dots , \xi _N =n_N \right ) = \P \left ( \eta ^{\gamma}_1= n_1, \dots , \eta ^{\gamma} _N =n_N ~\big |~  \sum _{j=1} ^N  \eta ^{\gamma} _j = M \right ). 
 \end{equation}
  In the physical terminology, the distribution in \eqref{invariant} is known  as the \emph {canonical ensemble}, while the product of the $N$ distributions in \eqref{individual} is the \emph {grand canonical ensemble}. The factor $\gamma$ is the so-called  \emph {chemical potential}.

  If it holds, the \emph {principle of equivalence} of  canonical and grand canonical ensembles tells that if $\gamma$ is rightly chosen, namely,  such that for $\eta ^{\gamma}  _j$'s as in \eqref{individual}, 
  \begin{equation} \label{chemicalpot}
   \E \left (\sum _{j=1}^N \eta ^{\gamma}  _j \right ) ~=~M~, 
   \end{equation}
   then the finite dimensional marginals of the distribution \eqref{invariant} are well approximated, for large $N$ and $M$, by products of  distributions  in \eqref{individual}. In other terms, for $N$ and $M$ large, one can forget the conditioning $  \eta ^{\gamma}  _1+\cdots +\eta ^{\gamma} _N = M$ provided that  $\gamma$ is rightly tuned, so that the total mean  of the $N$  free variables $\eta^{\gamma}_j$'s is $M$. 
   
   \medskip
   The following lemma proves strict monotonicity of  the left-hand side in  \eqref{chemicalpot}, ensuring uniqueness of $\gamma$ solving this equation. Existence of such a $\gamma$  holds, by continuity, provided that  for at least one $j$ (such that, necessarily, $Z_j(\gamma )$ has radius of convergence $\gamma ^*$)
   \[ \lim _{ \gamma \to \gamma ^*} \E \left ( \eta ^{\gamma}  _j \right ) = + \infty \]
  where $\gamma ^*$ is the minimal radius of convergence of all  the power series $Z_j(\gamma )$.
  
  \begin{lemma} \label{gamma_unique}
   Let $(X^{\gamma})_{0 < \gamma < \gamma ^*}$ be a family of $\, \N$-valued random variables with distributions
  \[ \P(  X ^{\gamma}  = n) = \frac{ \gamma ^{n} \phi (n) }{Z(\gamma)} \qquad \text{ for }\,  n \in \N,\]
  where $\phi$ is any nonnegative  function on $\N$ that is non-zero for  at least two points,
  \[ Z(\gamma) =   \sum _{n=0}^{\infty} \gamma ^{n} \phi (n), \]
  and $\gamma ^*$ is the radius of convergence of the series $Z(\gamma)$.
  Then $\E(X^{\gamma}) <+ \infty$ for  all $\gamma  \in ]0, \gamma ^*[$ and  the mapping $\gamma  \in ]0, \gamma ^*[  \ \longmapsto \E(X^{\gamma}) $ is  increasing.

  \end{lemma}
  
  \begin{proof}
  Both series $\sum _{n=0}^{\{} n \gamma ^{n} \phi (n)$  and  $\sum _{n=0}^{\infty} n ^2 \gamma ^{n} \phi (n)$ have same radius of convergence $\gamma ^*$ as $Z( \gamma)$, and $\E(X^{\gamma})$ is given for all  $\gamma  \in [0, \gamma ^*[$ by
  \[ \E(X^{\gamma}) = \frac {\sum _{n=0}^{\infty} n \gamma ^{n} \phi (n)}{ Z( \gamma)}.\] 
  
  Differentiating with respect to $\gamma \in \, ]0, \gamma ^*[$ gives 
  \[\frac{\partial \E(X^{\gamma}) }{ \partial  \gamma} = \gamma ^{-1} \left ( \frac {\sum _{n=0}^{\infty} n^2 \gamma ^{n} \phi (n)}{ Z( \gamma)} - \left (  \frac {\sum _{n=0}^{\infty} n \gamma ^{n} \phi (n)}{ Z( \gamma)} \right ) ^2 \right ) =  \gamma ^{-1}.  \, Var X^{\gamma} >0 \] 
  from assumption that $\phi(n) >0$ for at least two values of $n$, ensuring that all variables $X^{\gamma}$ for $0 < \gamma < \gamma ^*$ are non a.~s.~constant.
  
  \end{proof}
  
  From now on, we will denote $\eta _1, \dots , \eta _N$ the variables  $ \eta ^{\gamma}  _1, \dots ,  \eta ^{\gamma}  _N$ associated with the $\gamma$ solving \eqref{chemicalpot}, if it exists.  They will be referred as the ``free variables".
  
  \bigskip
  The Local limit Theorem is a classical tool for proving that equivalence of ensembles holds (\cite{Dobrushin-1}, \cite{Fayolle-7}, \cite{Khinchin-1}, \cite{Kipnis-1}). One can assume without loss of generality that the finite dimensional distribution of interest is that of $(\xi _1, \dots , \xi _K)$ for some $K \ge 1$, and   write  for any $(n_1, \dots , n_K) \in \N^K$:  
    \begin{multline} \label{key}
  \P \left (\xi _1= n_1, \dots , \xi _K =n_K \right ) = \P \left (\eta    _1= n_1, \dots , \eta   _K =n_K ~\big |~  \sum _{j=1} ^N \eta   _j = M \right ) \\
   =   \P \left (\eta    _1= n_1, \dots , \eta    _K =n_K \right ) \frac{\P  \left (\sum _{j=K+1} ^N \eta   _j = M -  \sum _{j=1} ^K n_j \right )}{\P  \left (\sum _{j=1} ^N \eta   _j = M \right )}.
    \end{multline}
    where the last equality results from independence of  the $\eta   _j$'s.
    
   If the Local Limit Theorem holds for both families of  variables $(\eta _j~,~ 1 \le j \le N)$ and $(\eta _j ~,~ K+1 \le j \le N)$, as $N$ goes to infinity, then since  by choice of $\gamma$,  $\E( \eta _1 )+\cdots +\E( \eta   _N )=M$,  one gets, informally:  
  $\quad \displaystyle{\P  \left (\sum _{j=1} ^N  \eta   _j = M \right ) \approx  \frac{1}{ b_N \sqrt{ 2 \pi}}} \quad $  and  
   \[   \P  \left (\sum _{j=K+1} ^N \eta   _j = M -  \sum _{j=1} ^K n_j \right )  \approx  \frac{\exp \left (- \frac{1}{2 \left ( b_N^2 -\sum _{j=1} ^K   \sigma _j^2 \right )} \left ( \sum _{j=1} ^K (\E(\eta _j) - n_j) \right )^2   \right )}{ \sqrt{ 2 \pi ( b^2_N-  \sum _{j=1} ^K  \sigma^2 _j )}} \] 
   where  $ b_N^2 =    \sum_{j=1}^N  \sigma _j^2 \, $,  denoting $ \sigma _j$  the standard deviation of  $\eta _j$.
   
   Now if  both  $\sum _{j=1} ^K  \sigma^2 _j $ and $ ( \sum _{j=1} ^K (\E(\eta  _j) - n_j)  )^2$ are negligible with respect to $ b_N^2$, the right-hand side of the last approximation is  close to $(  b_N \sqrt{ 2 \pi})^{-1}$  and one gets  the expected equivalence
   \[ \P \left (\xi _1= n_1, \dots , \xi _K =n_K \right ) \approx \P \left (\eta _1= n_1, \dots , \eta    _N =n_N  \right ).\]

To make this formal,  we must  consider a sequence of networks. In what follows, this sequence will be indexed by the number $N$ of nodes,  thus adding a subscript  $N$ to all parameters and variables.  

\medskip
 Remark that if  the Local Limit Theorem applies to the complete vector of free variables $(\eta _1, \dots , \eta _N)$,  it provides an  equivalent for the partition function $Z$ in \eqref{invariant}, as function of (the implicit) $\gamma$. Indeed,  using  equations  \eqref{invariant}, \eqref{individual} and \eqref{conditionning}, 
 \[ Z  =  \gamma ^{-M}\,   \left (  \prod _{i=1}^N Z_i(\gamma) \right )   \,  \P \left ( \sum _{i=1}^N \eta _i = M \right ) \approx  (\gamma ^{M}  b_N \sqrt{ 2 \pi} ) ^{-1}\prod _{i=1}^N Z_i(\gamma). \]
  
   \medskip 
  For application to  vehicle-sharing systems which is the object of Section~\ref{sec5}, we need  consider only two types of nodes: single server  nodes (with finite or infinite  capacity) and infinite server   nodes with infinite capacity.  For a single server node $j$, when $n$ customers are present, the departure rate from $j$ is $g_j(n)$      given by
   \[ g_j(n) = \mu _j  \quad \text{ if } n \ne 0 \quad \text { and }  g_j(0)= 0, \]
  while for an infinite server node, $g_j$ is given by
   \[g_j(n) = \mu _j n   \quad \text{ for all } n \in \N ,\] 
   where in both cases,   $\mu _j >0 $ is the parameter of the exponential   services at $j$. 
   
   \smallskip
   Equation \eqref{individual} then tells  that the variable $\eta _j$ is 
   \begin{itemize}
  \item geometric with parameter  $ \gamma \,   \mu _j^{-1}\theta _j$, truncated  to $\{ 0, \dots , c_j\}$  if $c_j < \infty$, for a single server node $j$,
    \item Poisson with parameter $   \gamma \,  \mu _j^{-1}  \theta _j $ if node $j$ has infinitely many servers.
  \end{itemize} 
  We will first consider a model where all  nodes have  infinite capacity, and then a model where single server nodes have finite (uniformly bounded) capacity. The equivalence of ensembles for those two models will respectively be derived from  the local limit theorems stated in Propositions \ref{geompoisson} and \ref {trunkgeompoisson}  of Section \ref{sec2}.

\bigskip
\subsection{Networks with infinite capacity nodes.} \label{sunsec4.1}

\vspace{1mm} We here consider a sequence of standard closed Jackson  networks:  all nodes have  infinite capacities. The network numbered $N$  has $N$ nodes (labelled by $1, \dots, N$)  and $M_N$ customers. Nodes are divided into  two  disjoint sets  $\mathcal J_1^N$ and $\mathcal J_2^N$. Those  in $ \mathcal J_1^N$ are single server nodes and those in $\mathcal J_2^N$ are infinite servers nodes.  To avoid trivialities, the set $\mathcal J_1^N$  is assumed to be non empty.
 
  For each network, we add a subscript $N$ to all quantities already defined for a general Jackson network,  thus denoting 
  \begin{itemize}
  \item $P_N=(p_{ij,N})_{1 \le i,j \le N}$ the (irreducible) routing matrix,
  \item $\theta _N = ( \theta_{1,N}, \dots , \theta _{N,N})$ the invariant probability vector associated to $P_N$,
  \item $\mu _{j,N}$ the rate of exponential services at node $j$ for $1 \le j \le N$,
  \item $(\xi _{j,N})_{1\leq j\leq N}$ the stationary node-occupation random vector.
\end{itemize}
The product formula  \eqref{invariant} here writes
 \begin{equation} \label{invariant_geom_poisson}  \P(\xi _{1,N}= n_1, \dots , \xi _{N,N} =n_N) = \frac{1}{Z _N}~  \,  \prod _{j \in  \mathcal J_1^N} r _{j,N} ^ {n_j}    \prod _{j \in \mathcal J_2^N} \frac {r _{j,N}^ {n_{j}}}{ n_{j} \,  !}
 \end{equation}
 where $(n_1, \dots , n_N ) \in \mathcal S_{N,M_N}$ and 
 \[  r _{j,N}= \mu _{j,N} ^{-1} \theta _{j,N} \qquad \text{  for } \quad  1 \le j \le N .\]
 
\smallskip
 The parameter 
 $r_{j,N}$ is usually referred to as the \emph{utilization} of node $j$.  From Remark \ref{uptoaconstant}, one knows that those utilizations need only be defined up to a constant. For this reason, some authors  normalize the $r_{j,N}$'s in such a way  that for each $N$, the maximum  $\max \{ r_{j,N}, 1 \le j \le N \} $ is equal to $1$.   This normalization is crucial in \cite{Malyshev-6}, where it is fully part of the assumption that  the empirical distribution 
$ N^{-1} \sum _{i=1} ^N \delta _{r_{j,N}}$ converges to some probability measure. Here,  it  will not be necessary, since less stringent conditions are needed: The factor $\gamma$ will in some sense absorb the normalization factor.

.
\medskip

 For all $N$,  let $\gamma _N \in  [0,  1/ \max \{ r_{j,N} ,\, j \in \mathcal J_1^N\}   [\ $ be the (unique, from lemma \ref{gamma_unique})   solution  to  equation  \eqref{chemicalpot},  here given by  
 \begin{equation} \label{gamma1}
 \sum _{j \in \mathcal J_1^N} \frac{ \gamma _N r_{j,N}}{ 1-\gamma _N r_{j,N}}~ + ~ \sum _{j \in \mathcal J_2^N}  \gamma_N r_{j,N} ~=  ~M_N.
 \end{equation}
 (The left-hand side  of this equation clearly  goes from  $0$  to $+\infty$ as  $\gamma _N$ varies from   $0$ to $  1/\max \{r_{j,N}, \,j \in \mathcal J_1^N \}  $, ensuring existence of the solution.)
 
The free variables $ (\eta _{j,N})_{ 1 \le j \le N}  $  associated to  $(\xi _{j,N})_{1\leq j\leq N}$ are
independent  and  
   \begin{itemize}
  \item for $j \in \mathcal J_1^N$,  $ \eta _{j,N}$ is geometric with  parameter  $\gamma _N r _{j,N}$, 
  \item for $j \in \mathcal J_2^N$,  $ \eta _{j,N}$ is Poisson with  parameter  $\gamma _N r _{j,N}$,
  \end{itemize}
so that  $\sum _{j=1}^{N}\eta _{j,N}$ has mean $M_N$ (from equations \eqref{gamma1}) and  variance  
  \begin{equation} \label{variance1}
 b _N^2 ~=~ \sum _{j \in \mathcal J_1^N} \frac{ \gamma _N r_{j,N}}{ (1-\gamma _N r_{j,N})^2}~ + ~ \sum _{j \in \mathcal J_2^N}  \gamma_N r_{j,N} .
 \end{equation}

 \begin{theorem}\label{network1}
 Assume that  the following conditions are satisfied
 \begin{enumerate}
 \item $\displaystyle{\limsup _{N \to \infty}  \left (\gamma _N  \max _{j \in \mathcal J_1^N } \{ r_{j,N}\} \right ) ~<~1}$,\\
  \item $\displaystyle{\lim _{N \to \infty}  b_N= + \infty}$,
\vspace{1mm}
\item $K \ge 1$ is such that   $\displaystyle{ \lim _{N \to \infty} b_N^{-1} \! \sum _{j \in \mathcal J_2^N \cap \{ 1, \dots ,  K\}  }\gamma_N r_{j,N}  = 0}\,$.
 \end{enumerate}
Then for any $(n_1, \dots , n_K) \in \N^K$, 
 \begin{equation}\label {equivalence}
\lim _{N \to \infty} \frac { \P \left (\xi _{1,N}= n_1, \dots , \xi _{K,N} =n_K \right )}{ \P  (\eta _{1,N}= n_1  ) \dots  P ( \eta _{K,N} =n_K) } ~=~1.
\end{equation}

 \end{theorem}
 
 \begin{proof}
 We use \eqref{key}  (here with a second subscript $N$ for all random variables).
   From equations \eqref{gamma1} and \eqref{variance1},  $\sum _{j=1 } ^{N} \eta _i$ has mean $M_N$ and variance $ b _N^2$.
  
   Condition (1)  together with  inequalities  $ \gamma _N <  1/\max_ {j \in \mathcal J_1^N }r_{j,N}$ for all $N$  imply existence of some $\rho <1$ such that 
  \[ \forall N, \quad \forall j \in   \mathcal J_1^N, \qquad  \gamma _N r _{j,N} \le \rho. \]

Condition (2) then allows to  apply  Proposition \ref{geompoisson}  to  both families
 $(\eta _{j,N})_{1 \le j \le N}$   and $(\eta _{K+j,N})_{1 \le j \le N-K}$ (with respectively, $J(N)=N$ and  $J(N)=N-K$). 
 This is immediate for the first family.  As for the second family,  it results from
 \[ Var  \left ( \sum_{j=K+1}^{N} \eta _{j,N}\right ) = b_N^2 -  \sum _{j \in \mathcal J_1^N \cap \{ 1, \dots ,  K\}  }  \frac{ \gamma _N r_{j,N} }{(1-\gamma _N r_{j,N})^2}  - \sum _{j \in \mathcal J_2^N \cap \{ 1, \dots ,  K\}  }\gamma_N r_{j,N} , \] 
 that  $ Var  \left ( \sum_{j=K+1}^{N} \eta _{j,N}\right ) \sim b_N^2$ as $N$ goes to infinity, since    both sums in the right-hand side are negligible with respect to $b_N^2$: indeed  the first sum is bounded above by $K \rho / (1- \rho )^2$, and the second one is negligible with respect to $b_N$, hence to $b_N^2$,  from assumption (3).

  Then using \eqref{tll},
 $\    b_N \sqrt { 2 \pi} ~ \P  \left (\sum _{j=1} ^N \eta _{j, N} = M_N \right ) $  converges to $1$  and  
\begin{multline*}
    \sqrt { 2 \pi ( b^2_N-  \sum _{j=1}^K  \sigma^2 _{j,N} )  } \;  \P  \left (\sum _{j=K+1} ^N \eta _{j,N} = M_N -  \sum _{j=1} ^K n_j \right ) \\- \exp \left (- \frac{ \left ( \sum _{j=1} ^K (  m _{j,N} - n_j) \right )^2}{2 \left ( b_N^2 -\sum _{j=1} ^K    \sigma _{j, N}^2 \right )}   \right ) 
\end{multline*}
 converges to $0$  for any $n_1, \dots , n_K \in \N$, as $N$ goes to infinity,
  where for $1\leq j\leq K$,
  $   m_{j,N}=   \E( \eta _{j,N})$ and   $ \sigma _{j, N}^2 =  Var ( \eta _{j,N})$.

   Now, on one hand, $ b^2_N-  \sum  _{j=1}^K  \sigma^2 _{j,N}  =  Var  \left ( \sum_{j=K+1}^{N} \eta _{j,N}\right ) \sim b_N ^2$ as $N$ goes to infinity, as just shown. On the other hand, the argument of the exponential in the  last expression goes to zero since 
     \[\frac{1}{b_N} \sum _{j=1} ^K  | m _{j,N} - n_j | \le   b_N ^{-1} \sum _{j=1} ^K  n_j + K  b_N ^{-1} \rho /(1- \rho)      +    b_N ^{-1} \sum_{j \in \mathcal J_2^N \cap \{ 1, \dots ,  K\}  }  \gamma _N r _{j,N}  \]
   which goes to zero as $N$ tends to infinity, using  assumptions (2) and  (3).  Hence
    $$  \lim _{N \to \infty }   b_N \sqrt { 2 \pi} ~ \P  \left (\sum _{j=K+1} ^N \eta _{j,N} = M_N -  \sum _{j=1} ^K n_j  \right ) = 1.$$
   
   Equation  \eqref{equivalence}  follows,  using \eqref{key} together with independence of the $\eta _{j,N}$'s.
    \end{proof}

    \begin{rem}\label {hyp} (i) Note that, since $M_N \le  b_N ^2$  for all $N$, as can be seen from \eqref{gamma1} and \eqref{variance1}, condition (2) of Theorem~\ref{network1} is satisfied if $\  \lim_{N \to \infty} M_N = \infty$.
  
    (ii)  Assumption (3) is trivially satisfied  for $ \{ 1, \dots ,  K\} \subset  \mathcal J_1^N$.

        \end{rem}
        
        \medskip
\subsection{Networks with both finite and infinite capacity nodes.} \label{subsec4.2}

The  networks  now considered  differ from the previous ones on two main points.
 For all $N$,
\begin{enumerate}
\item  for $j = 1, \dots , N$, node $j$ has capacity $c_{j,N}$, which is  now \emph {finite} for $j \in   \mathcal J_1^N$, but  remains  infinite for  $j \in \mathcal J_2^N$;  we will denote $c_N = (c_{j,N}, 1 \le   j \le N )$,
\item  both $ \mathcal J_1^N$ and $ \mathcal J_2^N$ are assumed to be non-empty.
\end{enumerate}

Any of the dynamics described in the previous section can be considered:   blocking, blocking-rerouting or general state-dependent routing satisfying the setting of theorem \ref{state_dependent}. In the simple blocking  case, we assume  that for all $N$,  the \emph {routing matrix is reversible} with respect to its invariant distribution. Thus, in any of these situations, the queue-length process has stationary distribution  given by the same  formula \eqref{invariant_geom_poisson} as in the infinite capacity case,  here with state space $\displaystyle{ \mathcal S^{c_N}_{N,M_N}}$.

We keep the same notations   as before:  $r_{j,N}$  are the (non-normalized) utilizations, $ \mu _{j,N}$ the exponential service parameters,  $ \theta _N$ the invariant probability vector, either of the fixed routing matrix $P_N$ (for blocking or blocking/rerouting procedures),  or common to all state-dependent routings, in the setting of  theorem \ref{state_dependent}.

For all $N$, $\gamma _N$ is now uniquely defined  in $[0, + \infty[$ by
\begin{equation} \label{gamma2}
 \sum _{j \in \mathcal J_1^N} \frac{ \sum_{n=0}^{c_{j,N}} n (\gamma _N r_{j,N})^n}{ \sum_{n=0}^{c_{j,N}}  (\gamma _N r_{j,N})^n}~ + ~ \sum _{j \in \mathcal J_2^N}  \gamma_N r_{j,N} ~=  ~M_N, 
  \end{equation}
 or equivalently by
\[
\sum _{j \in \mathcal J_1^N} \left ( \frac{ \gamma _N r_{j,N}}{ 1-\gamma _N r_{j,N}} - (c_{j,N}+1) \frac{ (\gamma _N r_{j,N})^{c_{j,N}+1}}{ 1-(\gamma _N r_{j,N})^{c_{j,N}+1}} \right )~ + ~ \sum _{j \in \mathcal J_2^N}  \gamma_N r_{j,N} ~=  ~M_N \]
 with the following abuse:   term  $j$ of the first sum has to be replaced by $c_{j,N}/2$ if $\gamma _N r_{j,N}=1$. 
 
  Existence and uniqueness of $\gamma _N$ again result from lemma \ref{gamma_unique} and from  the fact that the second term in the left-hand side of equation \eqref{gamma2}  increases to infinity  with $\gamma _N$ (recall that $\mathcal J_2^N \neq \emptyset$).
  
  The free variables $ (\eta _{j,N})_{ 1 \le j \le N}  $  associated to  $(\xi _{j,N})_{1\leq j\leq N}$ are now  such that
   \begin{itemize}
  \item for $j \in \mathcal J_1^N$,  $ \eta _{j,N}$ is geometric with  parameter  $\gamma _N r _{j,N}$ truncated to $[0, c_{j,N}]$, 
  \item for $j \in \mathcal J_2^N$,  $ \eta _{j,N}$ is Poisson with  parameter  $\gamma _N r _{j,N}$.
  \end{itemize}
   Denoting by $  m _{j,N}$ and $   \sigma _{j,N}^2$ the mean and variance of $\eta _{j,N}$ (for $1 \le j \le N$),   equation \eqref{gamma2} rewrites  $\  M_N=  \sum_{j=1}^N  m _{j,N}$, while  $\,   b_N^2 = \sum _{j=1}^N  \sigma _{j,N}^2 $
or equivalently
  \[  b_N^2  = \sum _{j \in \mathcal J_1^N} \left ( \frac{ \gamma _N r_{j,N}}{( 1-\gamma _N r_{j,N})^2} - (c_{j,N}+1)^2 \frac{ (\gamma _N r_{j,N})^{c_{j,N}+1}}{ (1-(\gamma _N r_{j,N})^{c_{j,N}+1})^2} \right )~ + ~ \sum _{j \in \mathcal J_2^N}  \gamma_N r_{j,N}. 
 \]
 Here again,  term  $j$ of the first sum must be replaced by $c_{j,N} (c_{j,N}+2)/12$ when $\gamma _N r_{j,N}=1$.
 
 \begin{theorem}\label{network2}
  If  the following conditions are satisfied
 \begin{enumerate}
 \item there exists some $C < +\infty$ such that for all $ N \ge 1$ and $j \in \mathcal J^1_N$, $~c _{j,N} \le C,$ 
\item $\lim _{N \to \infty}  b_N = + \infty$,
\item $K \ge 1$ is such that   $\displaystyle{ \lim _{N \to \infty} b_N^{-1} \! \sum _{j \in \mathcal J_2^N \cap \{ 1, \dots ,  K\}  }\gamma_N r_{j,N}  = 0}\,$.
\end{enumerate}
 then \eqref{equivalence} holds for all $(n_1, \dots , n_K) \in \N^K$.

 \end{theorem}
 
 \begin{proof}
 The proof is the same as for Theorem \ref{network1}, here using Proposition \ref{trunkgeompoisson}. Condition (1) ensures  that all $ m_{j,N}$ and $ \sigma _{j,N}$ for $j \in \mathcal J_1^N$ are bounded by some constant (not depending on $N$), namely, by $C^2$  (since the corresponding $\eta _{j,N}$'s have values in $\{0, \dots , C\}$). 
 \end{proof}
 
  Note that $(ii)$ of Remark  \ref{hyp} still holds. But this is not the case for $(i)$: if $m_{\rho, c}$ and $\sigma ^2_{\rho , c} $ denote the mean and variance of some   truncated geometric random variable  $X_{\rho, c}$ with parameters $\rho$ and $c$,  the ratio $ \sigma  _{\rho,c} /m_{\rho,c} $ approaches   zero as $\rho$ goes to zero or to infinity, for fixed $c \ge 1$.

   \section{Application to bike-sharing systems}\label{sec5}
   
   The purpose of this section is to derive  practical results on the performance of  bike-sharing systems  such as  the Velib' network in Paris. Here, a (large) set of bikes  are distributed over a (large) number  of docking stations and offered for use to the population of the city.  Any user can take a bike at some station, and then ride to another station where he  returns the bike.   The payment process  requires that each bike be locked to a terminal,  so that each  station can accommodate only a given number of bikes.

    \bigskip
\subsection {Infinite capacity approach.} \label{subsec5.1}
   
  It is easily seen that such a system can be modeled - at least  as long as all stations remain below capacity - by a closed Jackson network with a fixed number of customers (here  the bikes), and both single server nodes (the stations) and  infinite servers nodes (the routes from one station to another). Indeed, assuming that the arrival process of users at any station  $a$ is  Poisson with  parameter $\mu _a$, then bikes leave station $a$ at rate $\mu _{a}$ (here, inter-arrival intervals between users act as service durations for the waiting bikes). Next,  assuming that ride durations  are independent and  exponentially distributed  with parameter $\mu_{r}$ for  route $r$, the number of bikes on route $r$  decreases from  $n$ to $n-1$ at rate   $\mu _{r} n$. 
   
 These remarks  lead to a model with infinite node capacities  proposed in \cite{Fayolle-7} and  \cite{George-1}, that is considered now. 
   Denote by
       $\mathcal J_1$   the set of stations, with $ J_1 =|\mathcal J_1 |$, and   $\mathcal J_2 \subset \{ [ij] ~ :  ~i, j \in \mathcal J_1\}$, with $J_2= |\mathcal J_2 | \le J_1 (J_1 -1)$, the set of possible routes: Here $[ij]$ denotes the route from $i$ to $j$. It is assumed that the bike moves  obey to some  statistics which is constant in time: that is, any user taking a bike at station $i$ has probability $q_{ij}$ to put it back at station $j$.  Stations  here have unlimited  parking capacity, as well as  routes can  accommodate as many riders as necessary. Then bikes/customers  alternatively occupy nodes  in $\mathcal J_1$ and  in  $\mathcal J_2$ and move  according to the following routing matrix $P$ on  $\mathcal J_1 \cup \mathcal J_2$:
       \begin{equation}\label{routingGX}
        \   \forall i,j \in\mathcal J_1, \  \    p_{i [ij]} = q_{ij} \   \text{ and }  \   p_{[ij]j} =1 \  ; \ \     p_{ab} =0  \quad \text{for other  } a,b \in \mathcal J_1 \cup \mathcal J_2 . 
        \end{equation}
        Here $ \mathcal J_2$ is assumed to include all $[ij]$ such that $q_{ij}>0$, so that \eqref{routingGX} well defines a routing process on $\mathcal J_1 \cup \mathcal J_2$. 
   All services (inter-arrival times between users at stations or  trip durations) are independent, exponentially distributed with parameter $\mu _a$ at node $a$ ($a = j$  or $[ij]$ for some $i,j \in  \mathcal J_1$) and independent of the routing processes. This model thus constitutes a standard closed Jackson network with $M$ customers and $N =J_1+J_2$ nodes.
        
       Note that  $P$ is irreducible if and only if the following two conditions are satisfied:
           \begin{enumerate}
           \item the matrix $Q=(q_{ij}) _{i,j \in \mathcal J_1}$ is irreducible,
           \item $\forall i,j \in \mathcal J_1, \quad  [ij] \in \mathcal J_2~ \Longleftrightarrow ~q_{[ij]}>0.$
           \end{enumerate}
           We assume that this is the case. Then, denoting by $\nu = (\nu_j)_{j \in \mathcal J_1}$ the unique $Q$-invariant distribution on $\mathcal J _1$, the routing matrix $P$ has invariant distribution $ \theta = (\theta _a)_{a \in \mathcal J_1 \cup  \mathcal J_2}$  given by
             \begin{equation} \label{vector_GX}
              \theta _j = \frac{1}{2} \nu _j   \quad \text{ for } j \in \mathcal J_1 \qquad \text {and } \qquad  \quad \theta _{[ij]} = \frac{1}{2} \nu _i  q_{ij} \quad \text{ for }  [ij] \in \mathcal J_2  . 
              \end{equation}
     
        \medskip

         The process  so defined fits the frame of paragraph \ref{classicalJackson} of Section \ref{sec3} and has stationary state  given by 
         \begin{equation} \label{invariant_GX} \P(\xi _{1}= n_1, \dots , \xi _{N}=n_N) = \frac{1}{Z }~  \,  \prod _{j \in  \mathcal J_1} r _{j} ^ {n_j}    \prod _{[ij] \in \mathcal J_2} \frac {r _{[ij]}^ {n_{[ij]}}}{ n_{[ij]} \,  !}  
                  \end{equation} 
         where    $(n_1, \dots , n_N) \in \mathcal S _{N,M}$ and 
               \[   \text{ for } j \in \mathcal J_1, \quad  r _{j} =  \frac{\nu _{j} }{2 \,  \mu  _{j}} \qquad \text{and } \quad \text{ for }  [ij] \in \mathcal J_2,  \quad  r _{[ij]}= \frac{ \nu _{i} q _{ij} }{ 2\,\mu _{[ij]}}.\]

      \smallskip
      This model does not wholly convey the complexity of real  bike-sharing systems, since it ignores the blocking  mechanism at saturated stations (which is the main problem for these networks).   Still,  it describes the   system up to the first blocking time. 
       In \cite{George-1}, asymptotic results are  provided for $M$ large, as $N$ is  fixed. Here Theorem \ref {network1}  is used   to describe this  network  as  both $M$ and $N$ are large.

           The practical problem  addressed is the following.  For a network with given stations and routes, how many bikes   and parking places should be offered ? It can  be assumed  that   from observation-based estimations, the demand of the users is well known: that is,   the frequencies at which users arrive at the different stations ($\mu _j$'s),  the popularities of the different routes ($q_{ij}$'s)  and  their  mean ride durations ($\mu_{[ij]}$'s). In other terms,  utilizations $r_a =\mu _a ^{-1} \theta _a   $, are known.  Theorem \ref {network1}  then  helps evaluating, through approximation by  geometric variables,  the probability that  some station is empty, or exceeds  some level of occupation, indicating how to adapt the  stations capacities and/to  the number $M$ of bikes offered.
         
         In order to derive results from Theorem \ref{network1}, we again consider a sequence of networks indexed by $N$, and as in section \ref{sec4}, add a subscript $N$ to all parameters defined above, rewriting  equation \eqref{invariant_GX} as  equation \eqref{invariant_geom_poisson}. 
                
         A simple way to ensure that condition (1) of Theorem \ref{network1} is satisfied is to
choose some small positive $\delta$ and set   \begin{equation} \label{chooseM}
       \gamma _N = \frac {1- \delta}{\max _{j \in \mathcal J_1^N} r_{j,N}} \quad \text{and} \quad     M_N~=~\sum _{j \in \mathcal J_1^N} \frac{ \gamma _N  r_{j,N}}{ 1-\gamma _N  r_{j,N}}~ + ~ \sum _{[ij] \in \mathcal J_2^N}  \gamma _N r_{[ij],N}  \, , 
          \end{equation}
so that  $\gamma _N$   solves equation \eqref{gamma1}.  Theorem \ref{network1}  then has  the following corollary.

\begin{cor}  Let $\delta \in [0,1]$ be fixed and for all $N$, define $\gamma _N$ and $M_N$ by \eqref {chooseM}.

(i) If all $r_{j,N}$ for $j \in \mathcal J^N_1$ are of the same order, that is if  
\[\liminf_{N \to \infty} \frac { \min _{j \in \mathcal J_1^N} r_{j,N}}{ \max _{j \in \mathcal J_1^N} r_{j,N}} >0\, ,\]
 then as $N$ goes to infinity,  the stationary queue-lengths at  stations get asymptotically independent,  and approximately geometric with parameters 
 $\gamma _N r_{j,N} = (1- \delta) \, r_{j,N} / \max _{i \in \mathcal J_1^N} r_{i,N}$  ($j \in \mathcal J_1^N$).

 (ii) If moreover,
\[\lim _{N \to \infty} \frac { \max _{[ij] \in \mathcal J_2^N} r_{[ij],N}}{  \sqrt {J_1^N} \max _{k \in \mathcal J_1^N} r_{k,N}}  = 0\, , \] 
then at stationarity, both  stations and routes become asymptotically independent, and the  level of occupation of route $[ij]$  is approximately Poisson with parameter $\gamma _N r_{[ij],N} = (1- \delta) r_{[ij],N} / \max _{k \in \mathcal J_1^N} r_{k,N}$.
\end{cor}  

\begin{rem} This corollary applies in particular under the following  natural set of assumptions, for which $\gamma _N$  and $M_N$ are both  of the order of $J_1^N$: 
\begin{itemize}
\item all $r_{j,N}$ for $j \in \mathcal J^N_1$ are of the order of $1/J^N_1$,
\item all $r_{[ij],N}$ for $j \in \mathcal J^N_2$ are of the order of $1/J^N_2,$
\end{itemize}
(as for example if all  $\mu _{j,N} $ and  $\mu _{[ij],N}$ are of the order of $1$,   $\nu _{j,N}$  of the order of $1/J^N_1$ and  $\nu_{i,N} q_{ij,N}$  the order of $1/J^N_2$. Recall that $\sum _ j \nu _{j,N} =  \sum _ {[ij]} \nu _{i,N}  q_{[ij],N}  =1$).
Indeed, $\sqrt  {J_1^N} / J_2^N  \le 1/ \sqrt {J_1^N} \to 0$ as $N \to \infty$, since  $J_1^N  \le J_2^N $ by irreductibility of $(q_{ij}^N)$, and     $N =J_1^N + J_2^N \le (J_1^N )^2$. 

\end{rem}             
         
    \begin{proof}
    (i) Due to (ii) of Remark \ref {hyp}, we need only check that condition (2) of Theorem \ref{network1} is satisfied (condition (1) holds by choice of $\gamma _N$). But here
    \[ b_N^2 \ge M_N \ge \sum _{j \in \mathcal J_1^N}  \gamma _N  r_{j,N} = (1- \delta) \frac {  \sum _{j \in \mathcal J_1^N}   r_{j,N}} {\max _{j \in \mathcal J_1} r_{j,N}}  \ge (1- \delta) \,  J_1^N \frac { \min _{j \in \mathcal J_1^N} r_{j,N}}{ \max _{j \in \mathcal J_1^N} r_{j,N}} \]
    which goes to infinity with $N$, since $J_1^N \ge \sqrt N$ and $\liminf_{N \to \infty} \frac { \min _{j \in \mathcal J_1^N} r_{j,N}}{ \max _{j \in \mathcal J_1^N} r_{j,N}} >0$.

    (ii) We  need to prove that  $\max _{[ij] \in \mathcal J ^N_2}  \gamma _N r_{[ij],N} /b_N$ goes to zero as $N$ goes to infinity, in order that (3) of Theorem \ref{network1} holds. But  as just shown, $b_N^2 \ge c  \, J_1^N$  for some positive $c$ and $N$ sufficiently large. Hence,  for large $N$,
 \[ \forall [ij] \in \mathcal J^N_2, \quad \gamma _N r_{[ij],N} /b_N =  (1-\delta)  \frac {r_{[ij],N}} {b_N \max _{k \in \mathcal J_1^N} r_{k,N}} \le \frac{1-\delta}{ \sqrt {c J_1^N} } \frac {r_{[ij],N}}{ \max _{k \in \mathcal J_1^N} r_{k,N}} ,\]
 which maximum over $[ij]$ goes to zero by assumption.   
    \end{proof}

 As a practical consequence,  under condition (i) of Corollary 5.1,  the probability 
  \begin{itemize}
 \item   for station $j$  to be empty is approximately $1- (1- \delta) r_{j,N} / \max _{j \in \mathcal J_1 ^N} r_{i,N}$, which is at least $\delta$, and exactly $\delta$  at the most loaded station, 
 \item that station $j$  has more than $n$ vehicles  is $\left[ (1- \delta) r_{j,N} / \max _{j \in \mathcal J_1} r_{i,N}\right] ^{n+1}$.
 \end{itemize}
  Capacity of station $j$ can then be fixed as the smallest $n$ such that  this probability  of  exceeding  level $n$ is less than some given $\epsilon$. 
  
  In the ideal situation  when all utilizations $r_{j,N}$ for $j \in \mathcal J_1^N$ are equal,   the above probabilities are the same for all stations. The probability for any station to be empty,  $\delta$,  is determined by the total number $M_N$ of bikes, given by equation \eqref{chooseM}. And  the   smallest capacity that sets the blocking probability  (at any station) below some given level $\epsilon$ is given by the integer part of  $ \log \epsilon / \log (1- \delta)  $.

     \bigskip
  \subsection { Finite capacity approach.}

     We  now directly consider networks with finite parking capacity (and still infinite capacity routes). Remark first that  the routing matrix of the  model of \cite{George-1}, described in Section \ref{subsec5.1}, is  not reversible ($p_{[ij]i} =0$ for $i\ne j$ while $p_{i[ij]} =q_{ij}$). Thus, the  derived model with finite capacity stations and \emph {simple blocking} procedure (as defined in Section \ref{subsec3.2}) is  not tractable,  its stationary distribution being unknown. (Anyway, this procedure would  not be realistic, making blocked users perform again the same ride as the one just completed).

     Still, two options are open for modeling such finite parking capacity networks.  One is to consider the same routing as in   \cite{George-1} together with the \emph {blocking and rerouting} policy described in Section \ref{sec3}. This choice turns out to be relevant for bike-sharing systems. Indeed,  since transitions occur only from stations to routes and from routes to stations,  any user blocked at his end-of-route station $i$ will  choose some route at random among the routes issuing from $i$, which well describes the reality. Note that in the standard blocking and rerouting procedure, the new route is chosen according to the same probabilities $q_{ij}$'s as for  new arrived users, which may seem unrealistic, or  approximative (the $q_{ij}$'s  need then  combine   behaviors of  new arrived and redirected users). We will discuss existence of other state-dependent routings that may be  tractable (i.e.~with product-form stationary state) and  credible for bike-sharing.

     The other option is to use  pure blocking dynamics for a simpler model,  in which  all the routes are aggregated into one unique 
node,  having  infinite many servers and infinite capacity. The transition probability from any station to this unique route is then equal to $1$, while the transition from this route to any station $j$ is given by some probability $q_j$, called the \emph {popularity} of station $j$. This basic model  is a simple, easy to handle, approximation of the network in \cite{George-1}, which  does not make account of the detailed movements of the vehicles. It has been introduced and analyzed in \cite{Heterogeneous}.     Note that the associate routing matrix  $P$ is reversible, so that the product form  \eqref{invariant}  holds. Indeed, numbering the $N$ nodes ($N-1$ stations and one route) so that the set of stations is $\mathcal J_1 = \{1, \dots , N-1\} $ and  the unique route is $N$,  then $P$ is given by 
      \begin{equation}\label{routingpop}
       p_{jN}= 1 \quad \text{ and } \quad  p_{Nj } = q_j  \quad \text{ for }  j \in \mathcal J_1 \  ;  \quad   p_{ij} = 0  \text{ otherwise}, 
      \end{equation}
  and the invariant probability $\theta = (\theta _1, \dots , \theta _N)$ by 
  \[  \theta _N = \frac{1}{2} \quad  \text{ and } \quad \ \theta _j = \frac{q_j }{2} \quad \text{ for } j = 1, \dots , N-1,  \]
  which clearly satisfies the reversibility condition \eqref{reversible}. 
  
  Denote by $c_j$, for $\ 1 \le j <N$, the finite capacity of station $j$.  
  The stationary state of the network is here given by
  \[\P(\xi _1= n_1, \dots , \xi _N =n_N) = \frac{1}{Z _N}~    \frac {r _{N}^ {n_{N}}}{ n_{N} \,  !}  \,  \prod _{j=1}^{N-1} r _{j} ^ {n_j}  \] 
 for  $n \in \N^N $ such that $\ n_1 + \dots +n_N= M\, $ and $\ n_j \le c_j\ $ for $\ 1 \le j <N$,        where
         \[   \text{ for } j \in \{1, \dots , N-1\}, \quad  r _{j} =  \frac{q _{j}}{2 \,  \mu  _{j}} \qquad \text{and   }   \qquad  r _{N} = \frac{ 1 }{ 2\,\mu _{N}}.\]
  
  Here, since there is one unique route, the simple \emph{blocking}  and  the \emph{blocking and rerouting}  procedures are actually undistinguishable,  both leading to the above product form.

  \begin{rem}
These simplified  set of nodes  and transition matrix can also be used in the  infinite capacity case,  instead of  those of  \cite{George-1}. 
  \end{rem} 
  
   In modeling a bike-sharing system by any of these two models, it seems realistic to suppose that all the capacities $c_j$ stay bounded as $N$ is large. Theorem \ref{network2} can then be applied.

    For the last model with one unique route, parameters $\mu _i$ of services (inter-arrivals of users and durations of rides)  should be assumed to be of the order of $1$, while the $q_i$'s should be of the order of $N ^{-1}$.  One gets $r_j$'s of the order of $N ^{-1}$ for $1 \le j < N$ and $r_N$ of the order of $1$. 
   
   Equation \eqref{gamma2} here rewrites
   \[M ~=~  \sum _{j =1}^{N-1} \left ( \frac{ \gamma  r_{j}}{ 1-\gamma  r_{j}} - (c_{j}+1) \frac{ (\gamma  r_{j})^{c_{j}+1}}{ 1-(\gamma  r_{j})^{c_{j}+1}} \right )~ + ~   \gamma r_{N}
 \]
 (with the same abuse  as in Section \ref{sec4}, for terms $j \in \{1, \dots , N-1\}$   such that  $ \gamma  r_{j}= 1$).  It can be used, as in  the infinite capacity case, to fix $M$ by choosing $\gamma$. The appropriate scale  is here $\gamma = \kappa N$ for  some positive  $\kappa$,  which gives asymptotically independent stations, with  truncated-geometric approximate distributions  with parameters $\kappa N r_j$ and $c_j$ ($1 \le j <N$).
 Indeed, since $ b_N^2 \ge   \gamma r_N$ which  is of  the order of $N$, the approximation \eqref {equivalence} of  Theorem \ref{network2} is here valid for $\{1, \dots , K \} \subset \mathcal J_1 = \{ 1, \dots , N-1\}$. 
 
 Note that nothing can be said about node $N$, since $b_N$ is of order $\sqrt N$ ($ \gamma r_N \le b_N^2 \le   \gamma r_N + (N-1) \max c_j$)  so that $b_N ^{-1}  \gamma r_N$ gets large with $N$. The unique route is thus not proved to be independent from the stations, nor to have some identified approximate distribution: In  physical terms,  this scaling of $M$ (through $\gamma = \kappa N$) corresponds to a supercritical regime, with condensation at node $N$ (that accomodates as many vehicles as the order of $M$).
 
 \smallskip
 As regards stations: for $1 \le j <N$, 
 \begin{itemize}
 \item $\ P( \xi _j = 0) \approx 1/\sum _{n=0}^{c_j} (\kappa N r_j)^n  $, which decreases with respect to $\kappa$
  \item$ \ P( \xi _j = c_j) \approx (\kappa N r_j) ^{c_j}/\sum _{n=0}^{c_j} (\kappa N r_j)^n  $, which increases with respect to $\kappa$.
  \end{itemize}
  
  If capacities are given, one can look for some $\kappa$ that gives a reasonable trade-off  between the previous probabilities. These estimations can also help defining the capacities. Indeed for fixed $\kappa$, the approximation $1/\sum _{n=0}^{c_j} (\kappa N r_j)^n $ of  $\  P( \xi _j = 0) $ decreases  as $c_j$   increases. This is also the case for $ (\kappa N r_j) ^{c_j}/\sum _{n=0}^{c_j} (\kappa N r_j)^n  $, that approximates $\  P( \xi _j = c_j) $, due to  Remark \ref{symmetry}.
  So  $c_j$  can be chosen as the smallest value that sets both quantities  below some given level $\epsilon$.  However  the first quantity cannot be decreased below $1/\sum _{n=0}^{\infty} (\kappa N r_j)^n$ (which is positive if $\kappa N r_j <1$).  Similarly, the second quantity  cannot get below $1/\sum _{n=0}^{\infty} (\kappa N r_j)^{-n} $ (positive if $\kappa N r_j >1$).  Here again, the ideal situation is when all $r_j$ are equal. This indeed allows to select $\kappa$ such that $\kappa N r_j = 1$ for all $j= 1, \dots , N-1$; so that the approximate values for $\ P( \xi _j = 0) $ and $\ P( \xi _j = c_j) $ approach zero for large $c_j$'s. But real networks generally do not satisfy this condition, so that  $\kappa$  can only be fixed such  that $\kappa N r_j =1$ for a group of stations with equal utilizations. Both  probabilities - that a station is empty, or saturated - can then be  arbitrarily reduced only for these given stations.  No choice of $\kappa$ and $c_j$ can be globally satisfactory.

  \medskip
    We will  now consider  models derived from that of \cite{George-1}, here  limiting  station capacities. We first investigate possible alternatives to the standard blocking and rerouting procedure - in which rerouting is ruled by  the same matrix $Q$ that rules  choice of destination.   The setup is the same as  described in  section \ref{subsec5.1}, except that each node $j$  in $\mathcal J_1$ has finite capacity $c_j$. The Jackson network dynamics is then modified, to avoid  overflow of capacity, through the following transitions and rates:
    For $n \in \mathcal  S_{N,M}^c$, where $c = (c_a) _{ a \in \mathcal J_1 \cup  \mathcal J_2}$, setting $c_{[ij]} = \infty$ for $[ij] \in  \mathcal J_2$,
     \begin{equation}  \label{newrerouting} \begin{array}{lll}
   n ~\longrightarrow ~  n -e_i + e_{[ij]}  & \text{ at rate } & \mu _i\, q_{ij}\,  \bf 1_{n_i > 0} \\
     n ~ \longrightarrow ~  n - e_{[ij]}+ e_j & \text{ at rate } & \mu _{[ij]}\,  n_{[ij]} \, \bf 1_{n_j < c_j}  \\
     n ~ \longrightarrow  ~ n - e_{[ij]}+ e_{[jk]} & \text{ at rate }  & \mu _{[ij]}\,  n_{[ij]} \, w_{ik}^{(j)}(n - e_{[ij]}) \, \bf 1_{n_j = c_j \, ,  \, n_{[ij]} >0} 
     \end{array}
     \end{equation}
     Here for all $m \in \mathcal  S_{N,M-1}^c$ and $j \in \mathcal J_1$, a Markovian transition matrix   $W^{(j)}(m) = (w_{ik}^{(j)}(m) )_{i,k \in \mathcal J_1\setminus \{j \}}$ on the set  $ \mathcal J_1\setminus \{j \}$  is given.
     The following result is a consequence of Theorem  \ref{state_dependent}. 
    
    \begin{prop}  If  for each $j \in \mathcal J_1$, all  matrices $W^{(j)}(m)$ for $m \in \mathcal  S_{N,M-1}^c$ solve the following set of equations, with unknown variable  $W =  (w_{ik})_{i,k \in \mathcal J_1\setminus \{j \}}$: 
    \begin{equation} \label{w} \forall k  \in \mathcal J_1  \  \text{ such that } \  k \neq j, \qquad \nu _j q_{jk} = \sum _{i \in \mathcal J_ 1 \setminus \{j \}} \nu _i q_{ij} w_{ik},  
    \end{equation}
then, the process with state space $\mathcal  S _{N,M}^c$  and  transitions defined in  \eqref{newrerouting}  has  stationary distribution given by  \eqref {invariant_GX}  for $(n_1, \dots , n_N) \in \mathcal  S _{N,M}^c$.     
    \end{prop}

    \begin{proof}
    It is easily seen that the rates in \eqref {newrerouting}  have the form in \eqref {newrates} with
    \[ \text{for }   h \in \N, \quad g_i(h)= \mu _i\,  {\bf 1}_{h >0}\  \  \text{for }   i \in  \mathcal J_ 1   \quad  \text{  and }  \quad g_{[ij]}(h)= \mu _{[ij]} h \   \text{ for }   [ij] \in  \mathcal J_ 2  \]
    and markovian transition matrix $P(m)$ on the set $A(m)= \{ j \in \mathcal J_1, m_j < c_j \} \cup \mathcal J_2$ given,  for $m \in \mathcal  S_{N,M-1}^c$,    by
    \[ \begin{array}{llll}   p_{j, [ij]}(m) & =  & q_{ij} &     \text{for } [ij] \in \mathcal J_2 \text{ such that } i \in A(m), \\
    p_{[ij],j}(m) & = &1& \text{for } [ij] \in \mathcal J_2 \text{ such that } j \in A(m), \\
     p_{[ij],[jk]}(m) & = &w_{ik}^{(j)}(m)  & \text{for } i,j,k \in \mathcal J_1  \text{ such that }   [ij] , [jk] \in \mathcal J_2 \text{ and } j \notin A(m), \end{array} \] 
     all other transitions having null rates.
     
     Moreover, it is straightforward to check that  for all $m \in \mathcal  S_{N,M-1}^c$,  the restriction  of  vector $\theta$ of \eqref {vector_GX} to the set $A(m)$ is  invariant with respect to  $P(m)$. This  results from  equations \eqref{w} together  with invariance of  $\nu$ with respect to  $Q$.  
     
     The proposition then results from Theorem \ref{state_dependent}.
     
      \end{proof}
    Equations \eqref{w} are clearly satisfied if $w_{ik}^{(j)} = q_{jk}$ for  all $i,j,k$, which is the only solution such that $w_{ik}^{(j)}$ depends only on $(j,k)$. In other terms,   the above dynamics generalize the  standard blocking and rerouting - for the Jackson network  in \cite{George-1} - in the sense that  redirection of users blocked  at end of route $[ij]$ may now take into account   their original station $i$.  Note that this excludes  natural  reroutings to stations in the neighborhood of $j$. Still, for each $j$, equations \eqref{w} have infinitely many solutions, as it appears by rewriting them as
        \[q_{jk} = \sum _{i \in \mathcal J_ 1 \setminus \{j \}} \widetilde q_{ji} w_{ik} \qquad \text{ or else } \qquad q_{j  .} = \widetilde q_{j .} \,  W, \] 
  where $\widetilde Q= (\widetilde q_{ij})=( \nu _i q_{ij} / \nu _j)$ is the time-reversed  matrix of $Q$ under its invariant vector $\nu$, and $q_{j .}$, $ \widetilde q_{j .} $ denote the $j$-th lines of $Q$ and $\widetilde Q$.
 No general explicit solution is available, save  for  $W ^{(j)} = (q_{jk})_{i \neq j, k \neq j}$. But in the special case when  \emph {$Q$ is reversible under} $\nu$, the last equation means that $W$ has invariant vector $q_{j.}$; so that $W=Id\, $ is a solution for all $j$.   Choosing this solution:  $w_{ik}^{(j)}(m) = {\bf 1}_{k=i}\  (i,k \neq j)$  for all $j$ and $m$, means that blocked users are sent back to their original station. It can be interesting to note that for each $j$ the set of solutions of \eqref{w} is convex, so that one can take  convex combinations - with coefficient possibly depending on $m$ - of different solutions. In the reversible $Q$ case,   the following rerouting is then  possible: blocked users at end of route $[ij]$  flip a coin, which probability of giving ``head" may depend on the current distribution $m$ of the other $M-1$ bikes, and according to the result,  either take route $[ji]$, or choose new route $[jk]$ with probability $q_{jk}$. 
 
 Other  solutions can be given in particular cases, as  if for example  $Q$ is uniform, that is,  $q_{ij}= (J_1-1) ^{-1}$ for all $i,j $ with $i \neq j$.   Here,  variants of the deterministic rerouting from $[ij]$ to $[ji]$ can be  mentioned, such as rerouting from  $[ij]$  to $[j i^*_{j,m}]$,  where $m$ is the state of the network excluding the blocked bike, and  for all $j, m$, 
 \begin{itemize}
 \item either  $i^*_{j,m}$  is deterministic, and $i  \mapsto i^*_{j,m}$ is a one-to-one mapping on $\mathcal J_1 \setminus \{j \}$,
 \item   or    $i^*_{j,m}$  is the first step of some symmetric random walk - which distribution may depend on $j$ and $m$ - on some graph with vertex set $\mathcal J_1$ and constant degree.
 \end{itemize}
 
  \medskip
   We  will now consider any  dynamics that leads to  stationary distribution given by \eqref{invariant_GX}, here on state space $S _{N,M}^c$.
     In order to use the asymptotic results of section \ref{subsec4.2}, we here again consider a sequence of networks indexed by $N$, and use the same notations as in the infinite capacity case of section \ref{subsec5.1}.
     
      A set of natural hypotheses is for example: 
   
    \smallskip 
    (H0) $\exists \, C \in \N$ such that $c_{j,N} \le C$ for all $j \in \mathcal J_1^N$,
    
     (H1)  $\exists\,  \mu _+ , \ \mu _- >0\ $ such that $\ \mu _- \le \mu _{a,N} \le \mu _+ \ $  for all $N\, $ and $\, a \in  \mathcal J_1^N \cup \mathcal J_2^N$,
    
  (H2)  $\exists  \, A >0$ such that $ \displaystyle{\max _{[ij] \in  \mathcal J_2^N} \nu _{i,N} q_{ij,N}  \le A / J_2^N}\ $ for all $N$. 
 
  \noindent (Recall that $ \displaystyle{ \sum _{[ij] \in \mathcal J_2^N}   \nu _{i,N} q_{ij,N} = 1/2}$).   
      Theorem \ref{network2} then has the following  corollary.
     
     \begin{cor} Assume that (H0), (H1) and (H2) hold and that $(M_N)$ satisfies
     \[ \lim _{N \to \infty} M_N = + \infty \qquad \text{ and } \qquad  \lim _{N \to \infty} (J_2^N)^{-2} \, M_N = 0. \]
     
          Then  as $N$  goes to infinity,  at stationarity, the different queue-lengths at stations and routes get asymptotically independent, with  respective approximate distributions
          \begin{itemize} 
          \item truncated-geometric  with parameters $c_{j,N}$ and $ \gamma _N r_{j,N}$ for station $j$, 
          \item Poisson with parameter $ \gamma _N r_{[ij],N}$ for route $[ij]$,
                   \end{itemize} 
     where $\gamma _N$ solves equation \eqref{gamma2}.  Moreover, $\gamma _N$ has same order of magnitude as $M_N$.
      
            \end{cor}
     
     \begin{proof} Equation \eqref{gamma2} together with  the elementary relations
     \[    m_{\rho, c} ~ =  ~ \frac{ \sum_{n=1}^{c} n \rho^n}{ \sum_{n=0}^{c}  \rho^n} ~ \le ~ c \rho \,   \frac{ \sum_{n=0}^{c-1}  \rho^n}{ \sum_{n=0}^{c} \rho ^n} ~ \le ~ c \rho,  \quad \text  { for } \rho >0 \text{ and } c \in \N,\]
     imply the following inequalities, using (H0) and (H1): 
  \[ \frac{ \gamma _N}{2 \mu _+}  \le    \sum _{[ij] \in \mathcal J_2^N} \gamma_N  r_{[ij],N}  \le    M_N 
     \le   ~       C    
   \sum _{j \in \mathcal J_1^N}  \gamma _N r_{j,N}  ~ + ~ \sum _{[ij] \in \mathcal J_2^N}  \gamma _N  r_{[ij],N}      \le   ~   \frac{ \gamma _N}{2 \mu _-} (C+1),\]
   that show that $\gamma _N$ is of the same order  as $M_N$.
   
   Now assumptions (1), (2) and (3) of Theorem \ref{network2} are satisfied, for any $K\ge 1$. Indeed,   (1) is equivalent to  (H0). As for (2),   
   \[ b_N^2 ~ \ge    \sum _{[ij] \in \mathcal J_2^N} \gamma_N  r_{[ij],N}  ~\ge~  \frac{ \gamma _N}{2 \mu _+}   \,  , \]
   where $ \gamma _N  \ge  2 \mu _- M_N/(C+1) $ goes to infinity with $N$. And  for (3), using the previous lower bound on $b_N$, together with (H2) and  inequality $\gamma _N \le 2 \mu _+ M_N$, 
   \[ b_N^{-1} \max _{ [ij] \in \mathcal J_2^N} \gamma _N r_{[ij],N} ~  \le  ~ \frac {  \sqrt { 2 \mu _+  \gamma _N} }{ 2  \mu _ - }  \max _{ [ij] \in \mathcal J_2^N}  \nu _{i,N} q_{ij,N}~ \le ~  A \, \frac {    \mu _+}{   \mu _ - }   \frac { \sqrt { M _N}}{ J_2^N},\]
   which tends to zero as $N$ goes to infinity by assumption.
   
   The proof is then complete, using Theorem \ref{network2}. 
     \end{proof}

       Corollary  5.2 provides simple explicit approximations that may help measuring the performance of  a real network. As an example, the \emph {total stationary rate of failure} is given  (here removing index $N$)  by
       \[ \tau =  \sum _{j \in \mathcal J _1} \mu _j \P( \xi _j = 0) ~ +  \sum _{[ij] \in \mathcal J _2} \mu _{[ij]} \E \left ( \xi _{[ij]} \, { \bf 1} _{ \xi _j = c_j} \right ) .  \] 
       Using asymptotic independence and approximate distributions as stated in   Corollary  5.2, one gets
     \[ \tau \approx   \sum _{j \in \mathcal J _1}\frac{ \mu _j}{\sum _{n=0}^{c_j} (\gamma r_j)^n } ~ +  \sum _{[ij] \in \mathcal J _2} \mu _{[ij]} \gamma r _{[ij]} \, \frac{ (\gamma r _j)^{c_j}}{\sum _{n=0}^{c_j} (\gamma r_j)^n }  =   \sum _{j \in \mathcal J _1}  \mu _j\frac{ 1 +   (\gamma r _j)^{c_j+1}}{\sum _{n=0}^{c_j} (\gamma r_j)^n } , \] 
     where we have used $ \sum _i \mu _{[ij]} r _{[ij]} =  \sum _i \nu_i q_{ij}/2 = \nu_j /2 = \mu _j r_j. $  
     
     \smallskip
      If all traffic parameters are known, one can then minimize $\tau$ over $\gamma$, which amounts to choosing the best possible $M$, since $M$ and $\gamma$ are related by the one-to-one relation \eqref{gamma2}. Since the optimal $\gamma$, and hence  $M$,  should be of order $J_1 $ (if $r_j$'s are of order  $1/J_1$), conditions of Corollary 5.2 are satisfied.

    \section{Appendix}\label{secA}

\emph {Proof of Theorem \ref{TCL}.} 

\medskip
Following the lines of the proof of the Lindeberg Central Limit theorem given in \cite{Billingsley3} (Theorem 27.2  p. 359), it is not difficult to show  that \eqref{uniform} is satisfied, for some  sequence $(A_N)$ converging to infinity, provided that the following reinforcement of the Lindeberg condition is satisfied: 
 There exists some sequence  of positive real numbers   $\varepsilon _N$ such that $ \  \lim _{N \to \infty}   \varepsilon _N = 0 \ $ and  
\begin{equation}\label{Lindeberg+}
   \quad  \lim _{N \to \infty}  \frac{1}{b_N^{2}} \sum _{j=1}^{J(N)} \E \left ( (X_{j,N} - m_{j,N})^{2} 
{\bf 1}_{|X_{j,N} - m_{j,N}| > \varepsilon _N b_N} \right ) =0. 
  \end{equation}
  
  Now it is easily proved that the Lyapunov condition \eqref{Lyapunov} implies  existence of such a sequence $(\varepsilon _N)$. Indeed,  assuming that \eqref{Lyapunov} is satisfied, then  the following inequality holds  for any  positive $\varepsilon$,
  \begin{multline*}  \frac{1}{b_N^{2}} \sum _{j=1}^{J(N)} \E \left ( (X_{j,N} - m_{j,N})^{2} 
{\bf 1}_{|X_{j,N} - m_{j,N}| > \varepsilon  b_N} \right ) \\  \le  \frac{1}{\varepsilon ^{ \delta}b_N^{2+ \delta}} \sum _{j=1}^{J(N)} \E \left ( |X_{j,N} - m_{j,N}|^{2+ \delta} \right ) 
  =    \frac{\alpha _N}{\varepsilon ^{ \delta}} 
  \end{multline*}
   where $\delta$ is as in  \eqref{Lyapunov} and  $\alpha _N \stackrel{def}{=}  b_N^{-(2+ \delta)} \sum _{j=1}^{J(N)} \E \left ( |X_{j,N} - m_{j,N}|^{2+ \delta} \right ) $, so that $\lim \alpha _N =0$.  It results that \eqref{Lindeberg+} holds for $\varepsilon_N$ defined as  $\varepsilon_N =  \alpha _N ^{ 1/(2 \delta)}$, which clearly satisfies  $\ \lim \varepsilon_N =0$.

\bigskip 
\emph {Proof of Theorem \ref{TLL}.} 

\smallskip
First note that \eqref{uniform}  implies  the following, apparently stronger,  property
\begin{equation}\label{uniform+}
 \lim _{N \to \infty} \sup _{\ | t |  \le  B_N}\left | \, e^{\frac{t^2}{2}} \E \left (  e^{ it \frac{S_N-a_N}{b_N}}\right ) - 1 \right | = 0,
  \end{equation}
for some sequence $(B_N)$ of positive real numbers converging to infinity.
Indeed, assuming that \eqref{uniform} holds and  setting $\beta _N =  \sup _{\ | t |  \le  A_N}\left | \,  \E \left ( e^{ it \frac{S_N-a_N}{b_N}}\right ) -  e^{-t^2/2} \right | $ so that $\lim \beta _N = 0$, equation \eqref{uniform+} is then satisfied by any sequence $(B_N)$ such that
\[ 0 <B_N \le A_N \quad \text{ and } \quad e^{B_N ^2 /2} \le \beta _N ^{-1/2} \quad \text { for all } N, \]
since, if these inequalities hold, then 
\[ \sup _{\ | t |  \le  B_N}\left | e^{\frac{t^2}{2}} \E \left (  e^{ it \frac{S_N-a_N}{b_N}}\right ) - 1 \right |  \le  e^{\frac{B_N ^2}{ 2}}  \sup _{\ | t |  \le  A_N}\left | \, \E \left (  e^{ it \frac{S_N-a_N}{b_N}}\right ) - e^{-\frac{t^2}{2}} \right |   \le     \beta _N^{1/2}. \]
The numbers $B_N = \min ( A_N,  \sqrt { - \log \beta _N})$ thus satisfy  \eqref{uniform+} together with \mbox{$\lim B_N= \infty$.}

\medskip

The proof is now standard: Using inverse Fourier transform, it results from
\[ \E(e^{it S_N}) = \sum  _{k \in \Z} \P(S_N=k) \, e^{itk} \qquad (t \in \R), \]
that  for any  $k \in \Z$,
 \[    \P(S_N =k) = \frac{1}{2 \pi} \int _{- \pi} ^{ \pi} e^{-itk} \,   \E(e^{it S_N}) \,dt
=  \frac{1}{2 \pi b_N}  \int _{- \pi b_N} ^{ \pi b_N} e^{-itz_N(k)} \,   \E \left (e^{it  \frac{S_N - a_N}{b_N}} \right ) \,dt,\] 
where we have set $z_N(k) = (k-a_N)/b_N$.

It can be assumed without loss of generality that $  B_N \le \pi b_N$ for all $N \ge 1$. Otherwise, replace the sequence $(B_N)$ by $(B_N \wedge \pi b_N)$, which still satisfies \eqref {uniform+} and goes to infinity, since $ \lim_N b_N = \infty$  by assumption (1). Then for all $N$ and $k$,
\begin{align*} & \left | 2 \pi b_N  \, \P(S_N=k)       - \sqrt {2 \pi} e^{-z_N^2(k)/2} \right |\\
&  =   \left | \int _{- \pi b_N} ^{ \pi b_N} e^{-itz_N(k)} \,  \E \left (e^{it  \frac{S_N - a_N}{b_N}} \right ) \,dt  -  \int _{- \infty} ^{ \infty } e^{-itz_N(k)- t^2/2} \,dt  \right |  \\
&= \Big  | \int _{- B_N} ^{ B_N}  e^{-itz_N(k)- t^2/2}   \left (e^{t^2/2} \, \E \left (e^{it  \frac{S_N - a_N}{b_N}} \right ) -1 \right ) dt  \\   
& \qquad + 
 \int _{B_N \le  |t| \le  \pi b_N} e^{-itz_N(k)}  \, \E \left (e^{it  \frac{S_N - a_N}{b_N}} \right ) dt     
-  \int _{ |t| \ge B_N} e^{-itz_N(k)- t^2/2} \,dt  \Big | \\
 &\le   \left ( \sup _{|t| \le B_N}  \left | e^{t^2/2} \,  \E \left (e^{it  \frac{S_N - a_N}{b_N}} \right ) -1  \right |   \right )  \int _{- \infty} ^{ \infty } e^{- t^2/2} dt     \\
& \qquad +\int _{ |t| \ge B_N}  \phi(t) dt    +  \int _{ |t| \ge B_N} e^{- t^2/2}   dt . 
\end{align*}
Here assumption (3) has been used to dominate the second term.
 Since the three terms of the last sum do not depend on $k$ and converge to zero as $N$ goes to infinity, the theorem is proved.

\medskip
\emph {Proof of Proposition  \ref{geompoisson}.} 

\smallskip
One can assume without loss of generality that all the Poisson parameters $\lambda _{j,N}$ ($N \ge 1, j \in \mathcal J^2_N$) are bounded above by some fixed number, say $1$ for example.  Indeed, if this is not the case, one can replace each Poisson variable $X_{j,N}$ with  $\lambda_{j,N} >1$ by the sum of $\lfloor \lambda_{j,N}\rfloor +1 $ independent Poisson variables with the same parameter $ \lambda_{j,N} (\lfloor \lambda_{j,N} \rfloor +1 )^{-1}  <1$.
This change does not affect $S_N$, $a_N$ and $b_N$  (only $J(N)$ is increased). Thus  the proposition will be true in the general case if it is proved to hold when all $\lambda _{j,N}$ are less than $1$. 

Let us now show that   conditions (2) and (3) of  Theorem \ref{TLL}  are satisfied (note that (1) coincides with (ii) of the proposition). To prove that (2) holds, we use Theorem \ref{TCL} and check that the Lyapunov condition  is satisfied with $\delta =1$. The following domination of the centered third moment of any positive random variable $X$  by its non centered third moment is useful: 
\begin{equation}\label{third}
\E \left (|X-\E(X)|^3\right ) \le \E(X^3).
\end{equation}
This results from the  inequality, due to positivity of $X$, and hence  of $m =\E(X)$,
$|X-m|^3 \le (X+m) (X-m)^2= X^3 - mX^2-m^2 X +m^3, $
   so that 
   $\E(|X-m|^3) \le \E(X^3) -m \E(X^2) \le \E(X^3)$. 
   
   Recall that the third moment of a geometric  random variable with parameter $\rho$ is given by $ ( \rho+4 \rho^2 + \rho ^3)/(1-\rho)^3,$ while that of a Poisson variable with parameter $\lambda$ is $\lambda ^3 +3 \lambda ^2 + \lambda$.
   Using the fact that all  $\rho _{j,N}$'s and $\lambda_{j,N}$'s are less than $1$, one gets 
\begin{multline*}\frac{1}{b_N^{3}} \sum _{j=1}^{J(N)} \E \left ( |X_{j,N} - m_{j,N}|^3\right )  \le  \frac{1}{b_N^{3}}  \left ( \max _{1 \le j \le J(N)} \frac { \E \left ( |X_{j,N} - m_{j,N}|^3 \right ) }{\sigma _{j,N}^2 } \right ) \sum _{j=1}^{J(N)} \sigma _{j,N}^2 \\
\le    \frac{1}{b_N}   \max _{1 \le j \le J(N)} \frac { \E \left ( X_{j,N} ^3 \right ) }{\sigma _{j,N}^2 } \\
 \le   \frac{1}{b_N}   \left ( \max _{j \in \mathcal J_1^N}  \frac{ 1+4 \rho_{j,N} + \rho_{j,N} ^2}{1-\rho_{j,N}}     + \max _{j \in \mathcal J_1^N}  (\lambda_{j,N} ^2 +3 \lambda_{j,N}  + 1) \right)   
\le  \frac{1}{b_N}   \left (  \frac{ 6}{1-\rho}     +5  \right)  
\end{multline*}
where for the last step, assumption (i) is used.  Then by (ii), the Lyapunov condition is satisfied with $\delta=1$.

Now to check (3) of Theorem \ref{TLL}, recall that the characteristic function of a  geometric random variable $X$  with parameter $\rho$ is given by
\[ \E \left (e^{itX} \right ) =  \frac{1 - \rho }{1 - \rho e^{it}}  \]
while if $X$ is Poisson  with parameter $\lambda$, then
  \[ \E \left (e^{itX} \right ) = e^{ \lambda (e^{it}-1)}.  \]  
  One easily derives that 
  \begin{align*}
 & \left | \E \left (e^{itX} \right )  \right |=  \frac{1 - \rho }{\sqrt {1 - 2 \rho \cos t + \rho ^2}} \\
&=  \frac{1 - \rho }{\sqrt {(1-  \rho )^2 +2 \rho (1- \cos t)}} = \left  ( 1+4 \frac{\rho}{(1- \rho )^2} \sin ^2 (t/2) \right ) ^{-\frac{1}{2} } 
\end{align*}
 in the first case, while   in the second
 \[   \left | \E \left (e^{itX} \right )  \right | = e^{ - \lambda (1- \cos t)} = e^{-2 \lambda \sin ^2 (t/2)}. \]
 Then, using the  convexity inequality  $| \sin (t/2) | \ge |t |/ \pi$ for $|t | \le  \pi$,   $\  \left | \E \left (e^{itS_N} \right )  \right | =$
 \begin{align*}\prod  _{j=1}^{J(N)}  \left | \E \left (e^{itX_{j,N} }  \right )  \right | \le \left ( \prod _{j \in \mathcal J_1^N}  \left  ( 1+4  \frac{t^2}{\pi ^2}\frac{\rho_{j,N} }{(1- \rho_{j,N}  )^2}  \right )  ^{-\frac{1}{2} }  \right )  \exp   \left ( -2  \, \frac { t^2}{\pi ^2 } \sum  _{j \in \mathcal J_2^N} \lambda_{j,N}  \right ) \\
 = \exp  \left [  -\frac{1}{2}  \sum _{j \in \mathcal J_1^N}  \log \left  (1+4  \frac{t^2}{\pi ^2}\frac{\rho_{j,N} }{(1- \rho_{j,N}  )^2}  \right )  -2\,  \frac { t^2}{\pi ^2 } \sum  _{j \in \mathcal J_2^N} \lambda_{j,N}  \right] \quad \text{ for } |t | \le  \pi .
 \end{align*}
  Due to concavity of  the  $   \log$,  the inequality 
  $$4  \frac{t^2}{\pi ^2}\frac{\rho_{j,N} }{(1- \rho_{j,N}  )^2} \le 4 \frac{\rho}{(1- \rho )^2} \stackrel{def}{=} \tau$$
 that holds for all $t \in [- \pi, \pi]$ and all $j \in \mathcal J_1^N$ gives
  \[  \log \left  (1+4  \frac{t^2}{\pi ^2}\frac{\rho_{j,N} }{(1- \rho_{j,N}  )^2}  \right ) ~\ge~ 4 \frac{ \log(1+\tau)}{\tau}  \frac{t^2}{\pi ^2}\frac{\rho_{j,N} }{(1- \rho_{j,N}  )^2}.\] 
One obtains
 \begin{multline*} \left | \E \left (e^{itS_N} \right )  \right | \le   \exp  \left (  -2 \frac{ \log(1+\tau)}{\tau}  \frac{t^2}{\pi ^2}  \sum  _{j \in \mathcal J_1^N} \frac{\rho_{j,N} }{(1- \rho_{j,N}  )^2}  -2\,  \frac { t^2}{\pi ^2 } \sum  _{j \in \mathcal J_2^N} \lambda_{j,N}  \right)  \\
 \le   \exp \left  (-2 \frac{ \log(1+\tau)}{\tau}  \frac{ b^2_N t^2}{\pi ^2} \right )
 \quad \text{ for } |t | \le  \pi, 
  \end{multline*}
 since  \  $ \displaystyle { \frac{ \log(1+\tau)}{\tau} \le 1}$. The condition (3) of Theorem \ref{TLL} is thus satisfied with  integrable  $\phi$  given by  $\displaystyle {\phi (t) =  \exp \left  (-2 \frac{ \log(1+\tau)}{\tau}  \frac{  t^2}{\pi ^2} \right )}.$

\bigskip 
\emph {Proof of Proposition \ref{trunkgeompoisson}.}

\medskip
As in the preceding proof, one can assume that all $\lambda_{j,N}$ are less than one, and then proceed to check conditions (2) and (3) of Theorem \ref{TLL}. Condition (2) is obtained, here again, by proving that the Lyapunov condition holds with $\delta =1$. This goes exactly  as in Proposition \ref{geompoisson}, except
 for the truncated geometric variables  $X_{j,N}$ for $j \in  \mathcal J^1_N$. Here instead of using \eqref{third}, we use the following  inequality  (since $|X_{j,N} - m_{j,N}|$  is dominated by $c_{j,N}$):
\[\E \left ( |X_{j,N} - m_{j,N}|^3\right )   \le  c_{j,N} \E \left ( [X_{j,N} - m_{j,N}]^2\right ) =c_{j,N} \sigma _{j,N}^2 \le C  \sigma _{j,N}^2 .    \]
This gives
\begin{multline*}
\frac{1}{b_N^{3}} \sum _{j=1}^{J(N)} \E \left ( |X_{j,N} - m_{j,N}|^3\right )  \le  \frac{1}{b_N}  \left ( \max _{1 \le j \le J(N)} \frac { \E \left ( |X_{j,N} - m_{j,N}|^3 \right ) }{\sigma _{j,N}^2 } \right )  \\
 \le   \frac{1}{b_N}   \left ( C     + \max _{j \in \mathcal J_1^N}  (\lambda_{j,N} ^2 +3 \lambda_{j,N}  + 1) \right )  
\le  \frac{1}{b_N}   \left ( C     +5  \right )
\end{multline*}
and  by (ii),  the last quantity goes to zero as $N$ goes to infinity.  The Lyapunov condition, and hence (2), is satisfied.

As for condition (3),  using  Lemma~\ref{domination-Gnedenko} gives 
 \begin{multline*} \left | \E \left (e^{itS_N} \right )  \right | \le   \exp  \left (   - 2  \kappa   \frac {   t^2}{\pi ^2 } \sum  _{j \in \mathcal J_1^N} \sigma ^2_{j,N}  -2\,  \frac { t^2}{\pi ^2 } \sum  _{j \in \mathcal J_2^N} \lambda_{j,N}  \right)  \\
 \le   \exp \left  (-2 \min(1,\kappa) \frac{ b^2_N t^2}{\pi ^2} \right )
 \quad \text{ for } |t | \le  \pi.
  \end{multline*}

\providecommand{\bysame}{\leavevmode\hbox to3em{\hrulefill}\thinspace}
\providecommand{\MR}{\relax\ifhmode\unskip\space\fi MR }
\providecommand{\MRhref}[2]{%
  \href{http://www.ams.org/mathscinet-getitem?mr=#1}{#2}
}
\providecommand{\href}[2]{#2}

\bibliographystyle{amsplain}

\end{document}